\documentclass[10pt,letterpaper]{amsart} 

\usepackage{color}
\usepackage[letterpaper, margin=1in]{geometry}
\usepackage{subfig}
\usepackage{multicol}
\usepackage{listings}
\usepackage{amsmath, amsthm, amssymb, wasysym, verbatim, bbm, xcolor, graphics}
\usepackage{url}
\usepackage{mathtools}
\usepackage[colorlinks,linkcolor=blue]{hyperref}

\usepackage{xfrac}
\usepackage{float}
\usepackage{todonotes}
\newcommand{\R}{\mathbb{R}}

\newcommand{\N}{\mathbb{N}}

\newcommand{\tu}{\widetilde{u}}
\newcommand{\hu}{\widehat{u}}

\newcommand\norm[1]{\left\lVert#1\right\rVert}
\newcommand{\Grad}{\nabla}
\newcommand{\Div}{\operatorname{div}}
\newcommand{\dom}{\Omega}
\newcommand{\Vh}{\mathbb{V}_h}
\newcommand{\Ph}{\mathbb{P}_h}

\newcommand{\Yh}{\mathbb{Y}_h}

\newcommand{\Uh}{\mathbb{U}_h}
\newcommand{\weak}{\rightharpoonup}

\newcommand{\weakstar}{\overset{\star}\rightharpoonup}
\newcommand{\PUh}{\mathcal{P}_{\Uh}}
\newcommand{\PPh}{\mathcal{P}_{\Ph}}
\newcommand{\PN}{\mathcal{P}_{N}}
\newcommand{\PVh}{\mathcal{P}_{\Vh}}

\newcommand{\Pl}{\mathcal{P}} 

\newtheorem{lemma}{Lemma}[section]

\newtheorem{theorem}[lemma]{Theorem}
\newtheorem{remark}[lemma]{Remark}

\newtheorem*{maintheorem*}{Main Theorem}
\theoremstyle{definition}{\newtheorem{definition}[lemma]{Definition}}

\allowdisplaybreaks

\numberwithin{equation}{section}

\title[Convergence Projection Method]{A convergence proof for a  finite element discretization of Chorin's projection method of the incompressible Navier-Stokes equations}
\date{\today}
\thanks{This work was supported in part by the U.S. Department of Energy, Office of Science, Office of Advanced Scientific Computing Research's Applied Mathematics Competitive Portfolios program under Contract No. AC02-05CH11231 and by National Science Foundation award DMS 2042454.}

\author[F. Weber]{Franziska Weber}
\address[Franziska Weber]{\newline Department of Mathematics \newline UC Berkeley \newline Evans Hall, Berkeley,  CA 94702, USA.}
\email[]{fweber@berkeley.edu}

    \subjclass[2020]{65M12; 65M60; 76D05; 76M10}

\begin{document}

\begin{abstract}
	We study Chorin’s projection method combined with a finite element spatial discretization for the time-dependent incompressible Navier–Stokes equations. The scheme advances the solution in two steps: a prediction step, which computes an intermediate velocity field that is generally not divergence-free, and a projection step, which enforces (approximate) incompressibility by projecting this velocity onto the (approximately) divergence-free subspace.
	We establish convergence, up to a subsequence, of the numerical approximations generated by this scheme to a Leray–Hopf weak solution of the Navier–Stokes equations, without any additional regularity assumptions beyond square-integrable initial data. A discrete energy inequality yields a priori estimates, which we combine with a new compactness result to prove precompactness of the approximations in $L^2([0,T]\times\dom)$, where $[0,T]$ is the time interval and $\dom$ is the spatial domain. Passing to the limit as the discretization parameters vanish, we obtain a weak solution of the Navier–Stokes equations.
	A central difficulty is that different a priori bounds are available for the intermediate and projected velocity fields; our compactness argument carefully integrates these estimates to complete the convergence proof.

\end{abstract}

\maketitle

\section{Introduction}

We consider the incompressible Navier–Stokes equations describing the motion of a homogeneous viscous fluid in a bounded domain $\dom \subset \R^d$, $d=2,3$, over a finite time interval $[0,T]$:
\begin{equation}
	\label{eq:NS}
	\begin{split}
		\partial_t u + (u\cdot\nabla)u + \nabla p &= \mu \Delta u + f,\quad (t,x)\in [0,T]\times\Omega,\\
		\Div u & = 0,\quad (t,x)\in [0,T]\times\Omega,
	\end{split}
\end{equation}
where $u=u(t,x):[0,T]\times\Omega \to \mathbb{R}^d$ is the fluid velocity, $p=p(t,x):[0,T]\times\Omega\to \mathbb{R}$ is the pressure, and $f:[0,T]\times\Omega\to \mathbb{R}^d$ is a given forcing term. The domain $\Omega\subset\mathbb{R}^d$, $d=2,3$, is a bounded Lipschitz domain, $\mu>0$ is the kinematic viscosity (inversely proportional to the Reynolds number), and $T>0$ is a fixed final time. We impose homogeneous Dirichlet boundary conditions on $u$ and prescribe initial data
\begin{equation*}
	u(0,\cdot)=u_0\in L^2_{\text{div}}(\dom),
\end{equation*} 
where $L^2_{\mathrm{div}}(\Omega)$ denotes the subspace of square-integrable divergence-free vector fields with vanishing normal trace on $\partial\Omega$. We assume $f\in L^2(0,T;L^2_{\mathrm{div}}(\Omega))$, though the analysis can likely be extended to $f\in L^2(0,T;(H^1_{\mathrm{div}}(\Omega))^*)$ with a careful treatment of the forcing approximation in the finite element space.

The mathematical and numerical analysis of the Navier–Stokes system continues to be a central focus in computational fluid dynamics, owing both to its fundamental role in engineering and physics and to the profound mathematical difficulties it poses.
	From a numerical perspective, the key difficulty lies in resolving the pressure–velocity coupling while ensuring the discrete velocity field satisfies the divergence-free condition, or an approximate version of it, crucial for physical fidelity and stability.
	A widely used approach to this problem are projection methods, first introduced by Chorin~\cite{Chorin1967,Chorin1968} and Temam~\cite{Temam1969}, which decouple the velocity and pressure computations by a fractional-step time discretization. At each time step, one first solves a nonlinear convection–diffusion problem to compute an intermediate velocity $\tu^{m+1}$ that does not necessarily satisfy incompressibility, followed by a projection step that enforces the divergence-free constraint to yield the corrected velocity $u^{m+1}$. This approach significantly simplifies the numerical implementation and reduces computational cost compared to fully coupled solvers.
	More precisely, the original Chorin projection scheme~\cite{Chorin1967,Chorin1968} consists of a prediction step
	\begin{equation}
			\label{eq:chorinprediction}
		\frac{\tu^{m+1}-u^m}{\Delta t}+(u^m\cdot \nabla)\tu^{m+1} = \mu \Delta \tu^{m+1} + f^m,
	\end{equation}
		followed by a projection step solving
	\begin{equation}
			\label{eq:projectioncontinuous1}
		\frac{u^{m+1}-\tu^{m+1}}{\Delta t}+\nabla p^{m+1} = 0,\quad
		\Div u^{m+1} = 0.
	\end{equation}
	An important variant is the incremental projection method~\cite{Goda1979,vanKan1986,Guermond1996,Shen1996}, which introduces pressure increments in the prediction step to improve the pressure accuracy:
		\begin{equation}
			\label{eq:incrementalprediction}
		\frac{\tu^{m+1}-u^m}{\Delta t}+(u^m\cdot \nabla)\tu^{m+1} +\nabla p^m = \mu \Delta \tu^{m+1} + f^m,
		\end{equation}
			with the projection
	\begin{equation}
			\label{eq:projectioncontinuous2}
		\frac{u^{m+1}-\tu^{m+1}}{\Delta t}+\nabla (p^{m+1} - p^m) = 0,\quad
	\Div u^{m+1} = 0.
	\end{equation}
	These projection methods have enjoyed enduring popularity in both practical computations and theoretical studies due to their conceptual simplicity and modularity. They allow for separate solvers for velocity and pressure and have been adapted to various spatial discretizations, including finite differences, finite volumes, and finite elements.
	
	Despite their long history, rigorous convergence analyses of projection methods to weak solutions of the Navier–Stokes equations, particularly in fully discrete settings and for nonsmooth data, have only recently gained traction. Classical analyses often assume high regularity of the exact solution or are limited to semi-discrete (time-discrete) schemes~\cite{Guermond1996,Guermond1998,Hou1992,Hou1993,Achdou2000,E2002,Shen1996,Temam1968,Temam1977,Guermond2006}. In three spatial dimensions, where uniqueness and regularity are open problems, convergence to a Leray–Hopf weak solution (existence is guaranteed by Leray~\cite{Leray1934} and Hopf~\cite{Hopf1951}) is the natural benchmark.
	
	Fully discrete convergence results have been established over the last few years for finite difference schemes on arbitrary Lipschitz domains~\cite{Kuroki2020,Maeda2022}, and for finite volume methods on rectangular domains~\cite{Gallouet2023}. More recently, Eymard and Maltese~\cite{Eymard2024} proved convergence of an incremental projection finite element method with Poisson formulation of the projection step under the assumption of conforming finite element spaces, relying on a careful construction of a specialized interpolation operator. However, no convergence proof existed so far for the original Chorin projection method discretized via finite elements using the Darcy-type projection formulation. 
	
	Our work fills this gap by providing a proof of convergence for a fully discrete finite element scheme based on the original Chorin projection method with a Darcy formulation of the projection step. Our analysis covers Lipschitz domains in two and three dimensions with minimal regularity assumptions on the initial data and forcing. 
	The main difficulty arises from the lack of sufficient bounds on the pressure approximation and the intermediate velocity $\tu$ not being divergence free. So we can only obtain a bound on the time derivative of the approximation of $\tu$ in the dual space of $H^1_{\Div}\cap H^k(\dom)$ (where $k\in \N$ is sufficiently large and $H^1_{\Div}(\dom)$ denotes the space of almost everywhere divergence free $H^1_0(\dom)$ functions). Therefore, classical compactness results such as the Aubin-Lions-Simon lemma cannot be applied directly. To overcome this, we take a new approach to precompactness that carefully combines the a priori estimates on the approximation of the intermediate velocity $\tu$ with the approximately divergence free property of the approximation of the corrected velocity $u$. This way, we achieve strong precompactness in $L^2([0,T]\times\dom)$ for the approximations of $\tu$ and $u$.
	This new approach represents the core technical advance of our paper and could be of independent interest for other fractional step methods or multiphysics systems with split variables. Utilizing this tool, we show convergence (up to a subsequence) of the finite element approximations to a Leray–Hopf weak solution of the Navier–Stokes equations as both the spatial mesh size and time step tend to zero.
	
	From a numerical perspective, the Darcy formulation of the projection step we employ is attractive due to its direct variational form and compatibility with mixed finite element spaces. It is also convenient for coupling with porous media flow or multiphase flow solvers where Darcy equations naturally arise. Our analysis thus opens the door to rigorously justified finite element simulations of unsteady incompressible flow using this formulation, which had not been previously analyzed in the case of non-smooth solutions.
	We also revisit the incremental projection scheme analyzed in~\cite{Eymard2024}, showing that the convergence proof can be simplified and generalized to a wider class of (higher-order) finite element pairs without the need for specialized interpolation operators. This makes the approach more accessible and broadens its applicability. 
	
	Finally, it is important to emphasize that all projection methods considered here are first order in time. While they remain widely used due to their simplicity and robustness, extending rigorous convergence proofs for possibly non-smooth solutions to higher-order-in-time projection methods, such as for example~\cite{Bell1989}, remains an open and challenging problem, especially in fully discrete contexts. Such extensions would be highly valuable for increasing accuracy in long-time simulations of incompressible flows.
	
	The remainder of this paper is organized as follows. In Section~\ref{sec:prelim}, we introduce the necessary notation and recall the notion of Leray–Hopf weak solutions. Section~\ref{sec:num} presents the fully discrete finite element scheme for the Chorin projection method and establishes key a priori estimates. Section~\ref{sec:convergence} contains the proof of convergence to a weak solution. 
	In Section~\ref{sec:poisson}, we adapt the convergence proof of the incremental projection scheme from~\cite{Eymard2024} to a more general setting. Auxiliary results are collected in the Appendix.
	
\section{Preliminaries}\label{sec:prelim}
We start by introducing notation that will be used frequently in the following, and the definition of Leray-Hopf solutions for~\eqref{eq:NS}.

\subsection{Notation}
We denote the norm of a Banach space $X$ as $\|\cdot\|_X$ and its dual space by $X^*$.  We will denote $L^p$ spaces (for example, $L^2(\dom)$ for square integrable functions defined over $\dom$), Sobolev spaces and Bochner spaces in standard ways, and will not distinguish between scalar, vector-valued and tensor-valued function spaces when it is clear from the context. In particular, we use $L^p(0,T; X)$ to denote the space of functions $f:[0,T]\to X$ which are $L^p$-integrable in the time variable $t\in [0,T]$ and take values in   $X$.  The inner product on $L^2$ will be denoted by $(\cdot, \cdot)$.

We denote for a vector field $u:\dom\to\R^d$  the divergence $\Div u = \sum_{j=1}^d \partial_j u^j$ and the gradient $\Grad u = (\partial_i u^j)_{i,j=1}^d$.

We will use the subscript $\Div$  to indicate the divergence-free vector spaces, for example, 
\begin{equation*}
\label{eq:sigma_space}
\begin{aligned}
C_{c,\Div}^\infty(\dom)=\{{\phi}\in C_c^\infty(\dom); \Div {\phi}=0\},&\quad L^2_{\Div}(\dom)=\{ {\phi}\in L^2(\dom):\Div  {\phi}=0,  {\phi}\cdot {n}|_{\partial\dom}=0\}=\overline{C_{c,\Div}^\infty(\dom)}^{L^2(\dom)},\\
&\quad H^1_{ \Div}(\dom)=H^1_0(\dom)\cap L^2_{\Div}(\dom),
\end{aligned}
\end{equation*}
where we used $n$ to denote the outward normal vector of the domain $\dom$.
\begin{equation*}
	L^2_0(\dom)= \{\phi\in L^2(\dom); \, \int_{\dom} \phi dx =0\}
\end{equation*}
will be used to denote the space of $L^2$-functions that have zero average over the domain.
We denote the Leray projector by $\Pl: L^2(\dom)\to L^2_{\Div}(\dom)$, which is an orthogonal projection induced by the Helmholtz-Hodge decomposition $ {f}=\nabla g+ {h}$ for any $ {f}\in L^2(\dom)$~\cite{Temam1977}. Here, $g\in H^1(\dom)$ is a scalar field, and $ {h}\in L^2_{\Div}(\dom)$ is a divergence-free vector field. This decomposition is unique in $L^2(\dom)$. Then for all $ {f}\in L^2(\dom)$, it holds that $\mathcal{P}{f}={h}$.  

Next, we define  Leray-Hopf solutions for the incompressible Navier-Stokes equations which were introduced by Leray~\cite{Leray1934} for $\dom=\R^d$ and Hopf~\cite{Hopf1951} for bounded domains. 
\begin{definition}[Leray-Hopf solutions of~\eqref{eq:NS}]
	\label{def:weaksol}
Assume that $u_0\in L^2_{\Div}(\dom)$ and $f\in L^2(0,T;L^2_{\Div}(\dom))$.
	Let $u:[0,T]\times\dom\to \R^d$ be a vector field that is divergence-free, i.e., $\Div u=0$ for almost every $(t,x)\in [0,T]\times\dom$, and that satisfies
	\begin{equation}
	\label{eq:regularity}
\begin{split}
	&u\in L^\infty(0,T;L^2_{\Div}(\dom))\cap L^2(0,T;H^1_{\Div}(\dom)), \\
	&\partial_t u \in L^{4/3}(0,T;(H^1_{\Div}(\dom))^*), 
\end{split}
	\end{equation}
	and 
	\begin{equation}
	\label{eq:initdata}
	u(0,x) = u_0(x),\quad \text{a.e. }\, (t,x)\in [0,T]\times\dom.
	\end{equation}
Assume that $u$ satisfies 
 the distributional version of~\eqref{eq:NS}:
		\begin{equation}
		\label{eq:weakformu}
 \int_0^T  ({u},\partial_t v) dt +\int_0^T\int_\dom ((u\cdot\Grad )v) \cdot u dxdt
		=-\int_0^T (f,v) dt+\mu\int_0^T  (\Grad u,\Grad v )dt,
		\end{equation}	
for test functions $v\in C_c^\infty((0,T);C^\infty_{c,\Div}(\dom))$.
Moreover, let $u$ satisfy the energy inequality:
\begin{equation}
\label{eq:energyineq}
\frac12\norm{u(t)}_{L^2(\dom)}^2  +\mu\int_0^t\norm{\Grad u(s)}_{L^2}^2 ds
\leq \frac12\norm{u_0}_{L^2(\dom)}^2 + \int_0^t (f,u) ds,
\end{equation}
for a.e. $t\in [0,T]$ and be continuous at zero in $L^2(\dom)$, i.e.,
\begin{equation*}
	\lim_{t\to 0} \norm{u(t)-u_0}_{L^2(\dom)} = 0.
\end{equation*}
Then we call $u$ a Leray-Hopf solution of~\eqref{eq:NS}.
\end{definition}

\section{Numerical scheme}\label{sec:num}
Our scheme will be based on a linearly implicit discretization in time and LBB stable finite elements in space. To deal with the incompressiblility constraint, we will use a version of the projection method introduced by Chorin and Temam~\cite{Chorin1967,Chorin1968,Temam1969,Temam1977}. We start by describing the time discretization.

\subsection{Time discretization} 
We use a   linearly implicit time discretization based on Chorin's projection method~\cite{Chorin1967,Chorin1968}. This choice enables us to retain a discrete version of the energy inequality~\eqref{eq:energyineq}.
The idea of the projection method (or fractional step method) for the incompressible Navier-Stokes equations is to split the evolution of the Navier-Stokes equations~\eqref{eq:NS} into two steps: In the first step the velocity field is evolved according to the convection and the dissipation terms which may violate the divergence constraint. An intermediate velocity field $\tu$ is computed. In the second step the intermediate velocity field is projected onto divergence free fields. 

Specifically, we let $N\in \N$ the number of time levels and   discretize the time interval $[0,T]$ into time levels $t^m=m\Delta t$, $m=0,1,\dots, N$ where $\Delta t = T/N$. At each time level $t^m$, we seek discrete approximations $u^m(x)$ and $p^m(x)$ to $u(t^m,x)$ and $ p(t^m,x)$. We denote the initial data $(u^0, p^0)=(u_0,0)$. Then at every time level $m= 0,\dots, M-1$, given $(u^m,p^m)$,  we update $(u^{m+1}, p^{m+1})$ according to the following two steps:
\begin{enumerate}
	\item[{\bf Step 1}] (Prediction step): We compute $\tu^{m+1}$ through
		
		\begin{equation}\label{eq:step1semi}
		\frac{\tu^{m+1}-u^m}{\Delta t}+ (u^m\cdot\Grad)\tu^{m+1}= \mu \Delta \tu^{m+1}+f^{m}
		\end{equation}
	with boundary conditions 
	\begin{equation*}
	\left.\tu^{m+1}\right|_{\partial\dom}=0,
	\end{equation*}
	and where we have defined
	\begin{equation*}
		f^m(x)=\frac{1}{\Delta t}\int_{t^m}^{t^{m+1}}f(s,x) ds
	\end{equation*}
	\item[{\bf Step 2}] (Projection step): Next we define $(u^{m+1},p^{m+1})$ via
	\begin{subequations}
	\label{eq:projection}
	\begin{align}
	\frac{u^{m+1}-\tu^{m+1}}{\Delta t }&= -\Grad p^{m+1},\\
	\Div u^{m+1} & = 0.
		\end{align}	
	\end{subequations}
	with boundary condition
	\begin{equation}
	\label{eq:bcstep2}
	\left.u^{m+1}\cdot n\right|_{\partial \dom}=0.
	\end{equation}
\end{enumerate}
\begin{lemma}
	\label{thm:Solvability_scheme}
	Given an initial approximation $ u^0 \in H^1_0(\Omega) $,
	the numerical scheme \eqref{eq:step1semi}--\eqref{eq:projection} can be solved iteratively with $u^m\in H^1(\Omega)$ for every $m\in\mathbb{Z}$. Moreover, the numerical scheme~\eqref{eq:step1semi}--\eqref{eq:projection} is unconditionally energy stable and satisfies the semi-discrete energy dissipation law:
	\begin{multline}
	\label{eq:basic_semi-discrete_energy_law}
	\frac{1}{2}\norm{u^{M}}^2_{L^2}+\frac{1}{2}\sum_{m=0}^{M-1}\norm{u^{m+1}-\tu^{m+1}}_{L^2}^2
	+\frac{1}{2}\sum_{m=0}^{M-1}\norm{\tu^{m+1}-u^m}_{L^2}^2
	+\mu \sum_{m=0}^{M-1}\norm{\nabla\tu^{m+1}}_{L^2}^2\Delta t\\
	=\frac{1}{2}\norm{u^{0}}_{L^2}^2+\Delta t \sum_{m=0}^{M-1}(f^m,\tu^{m+1}),
	\end{multline}
	for all integers $M\in \left\{0,1,\dots,\frac{T}{\Delta t}\right\}$. 
\end{lemma}
We omit the proof as it is very similar to the upcoming proof  for the fully discrete scheme. 
\begin{remark}
	Using Gr\"onwall's inequality one can convert this to an energy bound, however with a constant that grows exponentially with time, i.e.,
	\begin{multline}
		\label{eq:basic_semi-discrete_energy_lawgr}
		\frac{1}{2}\norm{u^{M}}^2_{L^2}+\frac{1}{2}\sum_{m=0}^{M-1}\norm{u^{m+1}-\tu^{m+1}}_{L^2}^2
			+\frac{1}{2}\sum_{m=0}^{M-1}\norm{\tu^{m+1}-u^m}_{L^2}^2 \\
			+\mu \sum_{m=0}^{M-1}\norm{\nabla\tu^{m+1}}_{L^2}^2\Delta t  \leq C e^{c t^M} 	\left(\frac{1}{2}\norm{u^{0}}^2_{L^2}+\frac{\Delta t}{2} \sum_{m=0}^{M-1}\norm{f^m}_{L^2}^2\right),
		\end{multline}
		for constants $c,C>0$, c.f. Remark~\ref{rem:fneq0} below for the details on how to derive it. 
	\end{remark}
\subsection{Spatial discretization}\label{sec:spatial}
Next we describe the spatial discretization for the problem. 
We let $\mathcal{T}_h=\{K\}$ be families of conforming quasi-uniform triangulations of $\dom$ made of simplices with mesh size $h>0$. Here $h=\max_{K\in \mathcal{T}_h}\text{diam}(K)$. For simplicity, we assume that $\dom=\dom_h$, so that there is no geometric error caused by the domain approximation. We let $\Uh\subset H^1_0(\dom)^d$, and $\Ph\subset L^2_0(\dom)\cap H^1(\dom)$  be finite dimensional subspaces scaled by a meshsize $h>0$ that we will use for the spatial approximation of the velocity, and the pressure  respectively.  
The finite element spaces $\Uh$ and $\Ph$ need to satisfy the following two requirements: First, we require the $L^2$-orthgonal projections 
$\PUh:L^2(\dom)^d\to \Uh$, $\PPh:L^2(\dom)\to \Ph$ defined by
\begin{equation}\label{eq:L2projection}
	(u,\PUh v) = (u,v),\quad \forall \, u\in \Uh, \quad 	(p,\PUh q) = (p,q),\quad \forall \, p\in \Ph,
\end{equation}
to be stable on $H^1$ and $L^r$, $1\leq r\leq \infty$, and have  good approximation properties. In particular, we need for some constant $C>0$, independent of the mesh size $h>0$,
\begin{subequations}
	\label{eq:L2projproperties}
	\begin{align}
		\norm{\PUh v}_{L^r(\dom)}\leq C\norm{v}_{L^r(\dom)},&\quad \forall v \in L^r(\dom),\, r\in [1,\infty],\label{eq:Lr}\\
		\left|\PUh v\right|_{H^1(\dom)}\leq C\left|v\right|_{H^1(\dom)},&\quad \forall v\in H^1(\dom),\label{eq:H1}\\
		\norm{\PUh v - v}_{L^2(\dom)}\leq C h^s \norm{v}_{H^s(\dom)},&\quad \forall v\in H^s(\dom),\, s\in \left(\frac12,k+1\right],\label{eq:L2approx}\\
			\left|v-\PUh v\right|_{H^1(\dom)}\leq C h^{s-1} |v|_{H^s(\dom)},&\quad \forall v\in H^s(\dom),\, s\in [1,k+1],\label{eq:H1approx}
	\end{align}
\end{subequations}
where $k\in \N$ is the order of the finite element space, and we require the same properties for $\PPh$ for a possibly different $k$. Under the assumption of quasi-uniformity of the meshes,~\eqref{eq:L2projproperties} can be achieved for finite elements based on piecewise polynomial spaces such as
\begin{equation*}
	\Uh = \{v\in H^1_0(\dom)\, | \, \left.v\right|_K\in P_k(K),\, \forall \, K\in \mathcal{T}_h \},
\end{equation*}
where $P_k(K)$ is the space of polynomials of degreee $\leq k\in \N$.  See~\cite{Douglas1975} for~\eqref{eq:Lr},~\cite[Prop 22.19]{Ern2021} for~\eqref{eq:L2approx} and~\cite[Prop. 22.21]{Ern2021} for~\eqref{eq:H1} and~\eqref{eq:H1approx}. The condition of quasi-uniformity can be weakend in some cases, see Remark 22.23 in~\cite{Ern2021}.

Second, we assume that the pair $(\Uh,\Ph)$ satisfies the LBB (Ladyzhenskaya–Babu\v{s}ka–Brezzi) condition~\cite{Boffi2013}, meaning there exists $\beta^*>0$ independent of the mesh size $h>0$ such that  
\begin{equation}\label{eq:LBB}
\inf_{0\neq q\in \Ph}\sup_{0\neq v\in \Uh}\frac{(\Div v,q)}{\norm{\Grad v}_{L^2}\norm{q}_{L^2}}\geq \beta^*>0,
\end{equation}
where $\beta^*$ is independent of the discretization parameter $h>0$.
This is crucial for obtaining a stable approximation of the velocity-pressure pair of the Stokes equations (see e.g.~\cite{Boffi2013,Boyer2013,Guermond1998} and references therein).  For examples of pairs of finite element spaces $(\Uh,\Ph)$ that satisfy the LBB condition~\eqref{eq:LBB}, see for example~\cite{Girault1986,Ern2004}. The results that we present in the following apply to any pair of finite element spaces $(\Uh,\Ph)\subset H^1_0(\dom)^d\times L^2_0(\dom)\cap H^1(\dom)$ satisfying~\eqref{eq:L2projproperties} and~\eqref{eq:LBB}. One popular option are the Taylor-Hood finite elements, but also the MINI element by Arnold, Brezzi and Fortin~\cite{Arnold1984} satisfies the requirements. 

Next, we define the skew-symmetric trilinear form $b$ for functions $u,v,w\in H^1(\dom)$ with $u\cdot n=0$ on $\partial\dom$ as
\begin{equation}
\label{eq:defb}
b(u,v,w) = \int_{\dom} (u\cdot\Grad) v \cdot w dx +\frac12 \int_{\dom} \Div u v\cdot w dx.
\end{equation}
We observe that $b(u,w,v)= - b(u,v,w)$ and hence $b(u,v,v)=0$ for $u,v\in H^1(\dom)$ with $u\cdot n =0$ on $\partial\dom$ or $v=0$ on $\partial\dom$.  
  
 We start by approximating the initial data for $\tu$:
\begin{equation}
\label{eq:initdataapprox}
\tu^0_h = \PUh u_0, 
\end{equation} 
where $\PUh$ is the $L^2$-orthogonal projection.
Since $\tu_h^0$ may not be divergence free, we project it to an approximately divergence free function using the following projection step: Find $(u_h^0,p_h^0)\in \Uh\times\Ph$ such that
\begin{align*}
\left(\frac{u^{0}_h-\tu^{0}_h}{\Delta t }, v\right)&= (p^{0}_h,\Div v),\\
\left(\Div u^{0}_h,q\right)&=0,	
\end{align*}
for all $(v,q)\in \Uh\times \Ph$.

Then we define the following iterative scheme: 
\begin{enumerate}
	\item[{\bf Step 1}] (Prediction step):
	For any $m\geq 0$, given $ u_h^{m},\tu_h^m \in \Uh $, find $\tu_h^{m+1} \in \Uh $, such that for all $v\in \Uh $,
	\begin{equation}
		\label{eq:step1fully}
	\left(\frac{\tu^{m+1}_h-u^m_h}{\Delta t},v\right)+ b(\tu^m_h,\tu^{m+1}_h,v)+\mu(\Grad \tu_h^{m+1},\Grad v)= (\PUh f^{m}, v),
	\end{equation}
\item[{\bf Step 2}] (Projection step): Next we seek $(u^{m+1}_h,p^{m+1}_h)\in \Uh\times\Ph$ which satisfy for all $(v,q)\in \Uh\times\Ph$,
\begin{subequations}
	\label{eq:projectionfullydiscrete}
	\begin{align}
	\label{eq:projection1}
	\left(\frac{u^{m+1}_h-\tu^{m+1}_h}{\Delta t }, v\right)&= ( p^{m+1}_h,\Div v),\\
\label{eq:projection2}
	\left(\Div u^{m+1}_h,q\right)&=0,	
	\end{align}	
\end{subequations}
\end{enumerate}

\subsection{Solvability of the scheme}
Next, we show that the scheme~\eqref{eq:step1fully}--\eqref{eq:projectionfullydiscrete} is well-posed, i.e., given $(u^m_h,p_h^m)\in \Uh\times\Ph$, there is a unique $(u^{m+1}_h,p_h^{m+1})\in \Uh\times\Ph$ solving~\eqref{eq:step1fully}--\eqref{eq:projectionfullydiscrete}.  
\begin{lemma}
	\label{lem:solvability}
	For every $\Delta t, h>0$ and every $m=0,1,\dots$, given $(\tu_h^m,u^m_h,p^m_h)\in \Uh\times\Uh\times\Ph$, then there exists a unique $(\tu_h^{m+1},u^{m+1}_h,p^{m+1}_h)\in \Uh\times \Uh\times\Ph$ solving~\eqref{eq:step1fully}--\eqref{eq:projectionfullydiscrete}.
\end{lemma}
\begin{proof}
	We start by showing that if $(u_h^m,\tu_h^m,p^m_h)\in\Uh\times\Uh\times\Ph$, then there exists a unique $\tu^{m+1}_h\in \Uh$ satisfying~\eqref{eq:step1fully}, i.e., the first step of the algorithm is well-posed. To do so, we write~\eqref{eq:step1fully} as a linear variational problem: Find $\tu\in \Uh $ such that for all $v\in \Uh $,
	\begin{equation}
	\label{eq:elliptic}
	a^m( \tu ,v)= \mathcal{F}^m(v),
	\end{equation}
	where 
	\begin{equation}
	\label{eq:defa}
	a^m(\tu,v) =  	\left(\tu,v\right)+ \Delta t b(\tu^m_h,\tu ,v)+\Delta t\mu(\Grad \tu ,\Grad v) 
	\end{equation}
and
	\begin{equation}
	\label{eq:deffn}
	\mathcal{F}^m(v) = (u^m_h,v)+\Delta t(\PUh f^{m},v).
		\end{equation}
We show that $a^m$ is coercive and bounded on $\Uh$ and that $\mathcal{F}^m:\Uh\to \R$ is bounded. Then the Lax-Milgram theorem will imply well-posedness on $\Uh$ with the norm $\norm{v}_{\Uh}:=\norm{v}_{H^1_0}:= \norm{\Grad v}_{L^2}$.	To show coercivity, we plug in $v=\tu$:
	\begin{equation*}
		a^m(\tu,\tu) =  	\norm{\tu}_{L^2}^2+ \Delta t b(\tu^m_h,\tu ,\tu)+\Delta t\mu(\Grad \tu ,\Grad \tu)
= \norm{\tu}_{L^2}^2+ \Delta t\mu\norm{\Grad \tu}_{L^2}^2
 \geq c\norm{\tu}_{H^1_0}^2,
	\end{equation*}
	using the skew-symmetry of $b$. 
 This proves the coercivity on $\Uh$. For the boundedness, we have
	\begin{align*}
	\left| 	a^m(\tu,v)\right| \leq  &  	\norm{\tu}_{L^2}\norm{v}_{L^2}+ \Delta t \norm{\tu^m_h}_{L^4}\norm{\tu}_{H^1_0}\norm{v}_{L^4}+\Delta t\mu\norm{\Grad \tu}_{L^2}\norm{\Grad v}_{L^2}\\
\leq & \norm{\tu}_{L^2}\norm{v}_{L^2}+ \Delta t( \norm{\tu^m_h}_{L^4}+\mu)\norm{\tu}_{H^1_0}\norm{v}_{H^1_0}\leq C \norm{\tu}_{H^1_0}\norm{v}_{H^1_0}, 
	\end{align*}
	where we used the Sobolev embedding theorem and the Poincar\'e inequality. Here we also used that $\tu^m_h\in H^1_0(\dom)$ by assumption (and induction).
	 Hence $a^m$ is bounded on $\Uh$. 
	To show boundedness of $\mathcal{F}^m$ (for every $m$), we estimate:
	\begin{equation*}
	|\mathcal{F}^m(v)|\leq \norm{u^m_h}_{L^2}\norm{v}_{L^2}+\Delta t \norm{\PUh f^{m}}_{L^2}\norm{v}_{L^2}\leq \norm{u^m_h}_{L^2}\norm{v}_{L^2}+ \Delta t\norm{  f^{m}}_{L^2}\norm{v}_{L^2}.
	\end{equation*}
	 Thus, with the Poincar\'e inequality, we obtain boundedness of $\mathcal{F}^m$ on $\Uh$ since both $u^m_h$ and $f^m$ are in $L^2(\dom)$.
	We also note that $\Uh$ is dense in $H^1_0(\dom)$ and that these estimates hold for every fixed $h>0$.
	This implies the solvability of Step 1 of the algorithm.
	
	 The solvability of Step 2 follows from the choice of the LBB-stable pair $(\Uh,\Ph)$: Corollary 3.1 in~\cite{Guermond1996} implies the orthogonal decomposition $\Uh = \mathbb{H}_h \oplus B_h^\top (\Ph^1)$, 
	 where $B_h:\Uh\to \Ph$ is the operator defined by 
	 	\begin{equation*}
	 		(B_h v ,q ):=-(\Div v ,q ) = ( v,B_h^\top q ) ,\quad \forall\, q \in \Ph,\, v\in \Uh,
	 		\end{equation*}
	 and
	 where $\mathbb{H}_h = \ker (B_h)$ and $\Ph^1$ is the subset of $q\in \Ph$ for which $\norm{q}_{\Ph^1}=\norm{B_h^\top q}_{L^2(\dom)}$ is bounded. (Here $\norm{\cdot}_{\Ph^1}$ is induced by the inner product $(p,q)_{\Ph^1} = (B_h^\top p,B_h^\top q)$). Thus any $v\in \Uh$ can be decomposed into $v= w + B_h^\top p$, where $p\in \Ph^1$ and $w\in \Uh$ satisfies $B_h w=0,$ and $w$ and $B_h^\top p$ are orthogonal. This decomposition is unique. To see this, assume by contradiction that there are  $w_1,w_2\in \mathbb{H}_h$ and $p_1,p_2\in \Ph^1$ such that $v=w_1 + B_h^\top p_1 = w_2 + B_h^\top p_2$. We subtract the two equations from each other to obtain
	 \begin{equation}\label{eq:w1w2}
	 	0 = w_1 -w_2 + B_h^\top (p_1-p_2).
	 \end{equation}
	 $w_1-w_2\in \mathbb{H}_h$ and $p_1-p_2\in \Ph^1$, therefore $w_1-w_2$ and $B_h^\top(p_1-p_2)$ are orthogonal with respect  to $(\cdot,\cdot)$. Thus taking the inner product of~\eqref{eq:w1w2} with $w_1-w_2$, we obtain
	 \begin{equation*}
	 	(w_1-w_2,w_1-w_2)=0,
	 \end{equation*}
	 hence $w_1=w_2$ from which it also follows that $B_h^\top p_1 = B_h^\top p_2$. Since $\int p_1 dx =\int p_2 dx =0$, this also implies $p_1=p_2$, a contradiction. Now the projection step~\eqref{eq:projectionfullydiscrete} is exactly the decomposition of $\tu^{m+1}_h$ in $\mathbb{H}_h$ and $B_h^\top (\Ph^1)$: Indeed, we can write it as
	  \begin{equation*}
	  	\begin{split}
	  		\tu^{m+1}_h &= u^{m+1}_h +\Delta t B_h^\top p_h^{m+1},\\
	  		B_h u^{m+1}_h & = 0,
	  	\end{split}
	  \end{equation*}
	  i.e., $u^{m+1}_h\in \ker(B_h)=\mathbb{H}_h$ and $B_h^\top p_h^{m+1}\in B_h^{\top }(\Ph^1)$ with $\tu^{m+1}_h = u^{m+1}_h +\Delta t B_h^\top p_h^{m+1}$ and we already know this decomposition is unique (see also~\cite{Guermond1998}).
\end{proof}
\subsection{Energy stability}
Next, we show the discrete energy inequality which will result in a priori estimates on the approximations that will allow us to derive precompactness of the approximations. 

\begin{lemma}
	\label{lem:discenergyestimate}
	The approximations computed by the scheme~\eqref{eq:step1fully}--\eqref{eq:projectionfullydiscrete} satisfy
	\begin{equation}\label{eq:discenergybalance}
	E^M_h	 
	 +\frac12\sum_{m=0}^{M-1}\norm{\tu_h^{m+1}-u^m_h}_{L^2}^2 
	+\frac12\sum_{m=0}^{M-1}\norm{u^{m+1}_h-\tu^{m+1}_h}^2_{L^2}+\mu\Delta t\sum_{m=0}^{M-1}\norm{\Grad \tu_h^{m+1}}_{L^2}^2
 = E^0_h + \Delta t \sum_{m=0}^{M-1}(  f^m,\tu_h^{m+1}),
	\end{equation}	
	where
	\begin{equation*}
	E^M_h := 	\frac{1}{2}\norm{u^{M}_h}_{L^2}^2,
	\end{equation*}
	for any integer $M\in \left\{0,1,\dots, \frac{T}{\Delta t}\right\}$.
\end{lemma}
\begin{proof}
	To derive the energy estimate, we take as a test function $v=\tu^{m+1}_h\in\Uh$  in~\eqref{eq:step1fully} and use the simple identity $2a(a-b)=a^2-b^2+(b-a)^2$:
		\begin{multline*}
	\frac{1}{2\Delta t}\left(\norm{\tu^{m+1}_h}_{L^2}^2-\norm{u^m_h}^2_{L^2}+\norm{\tu_h^{m+1}-u^m_h}_{L^2}^2\right)+ b(\tu^m_h,\tu^{m+1}_h, \tu_h^{m+1})\\
= -\mu(\Grad \tu_h^{m+1},\Grad \tu_h^{m+1})+ (\PUh f^{m},\tu^{m+1}_h).
	\end{multline*}
	Using the skew-symmetry of $b$, we obtain
	\begin{equation}\label{eq:energytemp}
	\frac{1}{2\Delta t}\left(\norm{\tu^{m+1}_h}_{L^2}^2-\norm{u^m_h}^2_{L^2}+\norm{\tu_h^{m+1}-u^m_h}_{L^2}^2\right)
	= -\mu\norm{\Grad \tu_h^{m+1}}_{L^2}^2 + (\PUh f^{m},\tu^{m+1}_h).
	\end{equation}
	This looks promising, except we would prefer to have $u_h^{m+1}$ instead of $\tu^{m+1}_h$ in the kinetic part of the energy.
	Therefore, we take $v=u^{m+1}_h\in \Uh$ as a test function in~\eqref{eq:projection1}. Then since $u^{m+1}_h$ satisfies~\eqref{eq:projection2} for all $q\in\Ph$ and $p^{m+1}_h\in\Ph$, we see that the second term is zero and hence
	\begin{equation}\label{eq:utildebound}
	0=\left(\frac{u^{m+1}_h-\tu^{m+1}_h}{\Delta t},u^{m+1}_h\right)=\frac{1}{2\Delta t}\left(\norm{u^{m+1}_h}_{L^2}^2-\norm{\tu^{m+1}_h}^2_{L^2}+\norm{u^{m+1}_h-\tu^{m+1}_h}^2_{L^2}\right),
	\end{equation} 
	where we have used the identity $2a(a-b)= (a^2-b^2+(a-b)^2)$ once more.
Using this identity in~\eqref{eq:energytemp}, and the definition of the $L^2$-projection,~\eqref{eq:L2projection}, we obtain
\begin{multline}\label{eq:heart}
	\frac{1}{2\Delta t}\left(\norm{u^{m+1}_h}_{L^2}^2-\norm{u^m_h}^2_{L^2}+\norm{\tu_h^{m+1}-u^m_h}_{L^2}^2+\norm{u^{m+1}_h-\tu^{m+1}_h}^2_{L^2}\right)\\
= -\mu\norm{\Grad \tu_h^{m+1}}_{L^2}^2+ (  f^{m},\tu^{m+1}_h).
\end{multline}
Now, we can multiply by $\Delta t$, and sum over $m=0,\dots, M-1$ to obtain
\begin{multline*}
	\frac{1}{2}\norm{u^{M}_h}_{L^2}^2 +\frac12\sum_{m=0}^{M-1}\norm{\tu_h^{m+1}-u^m_h}_{L^2}^2 
+\frac12\sum_{m=0}^{M-1}\norm{u^{m+1}_h-\tu^{m+1}_h}^2_{L^2}+\mu\Delta t\sum_{m=0}^{M-1}\norm{\Grad \tu_h^{m+1}}_{L^2}^2
 \\
= \frac12\norm{u^0_h}^2_{L^2} +\Delta t \sum_{m=0}^{M-1} (  f^m,\tu^{m+1}_h)
\end{multline*}	
which is what we wanted to prove.	
\end{proof}
\begin{remark}\label{rem:fneq0}
	If $f\neq 0$, we can show an estimate on the energy using a discrete version of Gr\"onwall's inequality: 
	We estimate with Young's inequality
	\begin{equation*}
		\left|(f^m,\tu_h^{m+1})\right|\leq \frac{1}{2}\left(\norm{f^m}_{L^2}^2+\norm{\tu_h^{m+1}}_{L^2}^2\right),
	\end{equation*}
	and then use~\eqref{eq:utildebound},
		\begin{equation*}
		\left|(f^m,\tu_h^{m+1})\right|\leq \frac{1}{2}\left(\norm{f^m}_{L^2}^2+\norm{u_h^{m+1}}_{L^2}^2+ \norm{u^{m+1}_h-\tu^{m+1}_h}_{L^2}^2\right),
	\end{equation*}
	thus, we can estimate in~\eqref{eq:heart}  
	\begin{equation*}
		(1-\Delta t)E_h^{m+1} + (1-\Delta t)D_h^{m+1}\leq E_h^m + \frac{\Delta t}{2}\norm{f^m}_{L^2}^2,
	\end{equation*}
	where we denoted
	\begin{equation*}
		D_h^{m+1} = \frac12\norm{\tu_h^{m+1}-u^m_h}_{L^2}^2+\frac12\norm{u^{m+1}_h-\tu^{m+1}_h}^2_{L^2} + 
	\Delta t\mu\norm{\Grad \tu_h^{m+1}}_{L^2}^2.
	\end{equation*}
	Since $\frac{1}{1-\Delta t}\leq 1+2\Delta t$ for $\Delta t$ small enough, we get
	\begin{equation*}
		E_h^{m+1} + D_h^{m+1}\leq (1+2\Delta t)E_h^m + \frac{(1+2\Delta t)\Delta t}{2}\norm{f^m}_{L^2}^2,
	\end{equation*}
	which with an application of a discrete Gr\"onwall inequality, Lemma~\ref{lem:discretegronwall}, leads to
	\begin{equation*}
		E^M_h +\sum_{m=1}^{M}D_h^m\leq Ce^{2M\Delta t}\left(E_h^0+\frac{ \Delta t}{2}\sum_{m=0}^{M-1}\norm{f^m}_{L^2}^2\right).
	\end{equation*}
\end{remark}
\begin{remark}
	\label{rem:utildeL2bound}
	We also obtain from~\eqref{eq:utildebound}  that 
	\begin{equation}\label{eq:utildeL2bound}
		\frac{1}{2}\norm{\tu^M_h}_{L^2}^2\leq Ce^{2M\Delta t}\left(E_h^0+\frac{ \Delta t}{2}\sum_{m=0}^{M-1}\norm{f^m}_{L^2}^2\right),
	\end{equation}
	for any $M\in \N$.
\end{remark}
\section{Convergence analysis of the scheme}\label{sec:convergence}
Next, we will show that interpolations of the approximations defined by the scheme~\eqref{eq:step1fully}--\eqref{eq:projectionfullydiscrete} converge up to a subsequence to a weak solution of~\eqref{eq:NS}. For this purpose, we define the piecewise linear and piecewise constant interpolants in time:
\begin{alignat}{2}
	u_h(t)& = u_h^m+\frac{t-t^m}{\Delta t}(u^{m+1}_h-u^m_h),\qquad &t\in [t^m,t^{m+1}),\label{eq:defuh}\\
	\bar{u}_h(t) &= \tu_h^m,\qquad &t\in [t^m,t^{m+1}),\\
		\hu_h(t)& = \tu_h^{m+1}+\frac{t-t^m}{\Delta t}(\tu^{m+2}_h-\tu^{m+1}_h),\qquad &t\in [t^m,t^{m+1}),\\
	\widetilde{u}_h(t) &= \widetilde{u}_h^{m+1},\qquad &t\in (t^m,t^{m+1}],\\
	p_h(t)& = p_h^{m+1},\qquad & t\in (t^m,t^{m+1}],\\
	f_h(t)& = \frac{1}{\Delta t}\int_{t^m}^{t^{m+1}}\PUh f(s)ds,\qquad &t\in (t^m,t^{m+1}],\label{eq:defph}
\end{alignat}
for $m=0,1,2,\dots$ with
\begin{equation*}
	\tu_h(0)=u^0_h,\quad \text{and}\quad p_h(0)=p_h^0.
\end{equation*}
From the energy balance in the last section, Lemma~\ref{lem:discenergyestimate} and Remark~\ref{rem:fneq0}, we obtain the following uniform a priori estimates for the sequences $\{u_h\}_{h>0}$, $\{\bar{u}_h\}_{h>0}$, $\{\hu_h\}_{h>0}$, and $\{\tu_h\}_{h>0}$:
\begin{align*}
	&\{u_h\}_{h>0}\subset L^\infty(0,T;L^2(\dom)),\\
	&\{\bar{u}_h\}_{h>0},\{\widetilde{u}_h\}_{h>0}, \{\widehat{u}_h\}_{h>0}\subset L^\infty(0,T;L^2(\dom))\cap L^2(0,T;H^1_0(\dom)).
\end{align*}
These uniform estimates imply, using the Banach-Alaoglu theorem, that there exist weakly convergent subsequences, which, for the ease of notation, we still denote by $h\to 0$,
\begin{align}
	u_h\weakstar u,\quad 	\bar{u}_h\weakstar \bar{u},\quad \widetilde{u}_h\weakstar \widetilde{u},\quad \widehat{u}_h\weakstar \widehat{u},&\quad \text{in }\, L^\infty(0,T;L^2(\dom)),\label{eq:uhweakconv}\\
\bar{u}_h\weak \bar{u},\quad	\widetilde{u}_h\weak \widetilde{u},\quad 	\widehat{u}_h\weak \widehat{u},&\quad \text{in }\, L^2(0,T;H^1_0(\dom)).\label{eq:Hhweakconv}
\end{align}
In order to prove convergence of the scheme, we will need to derive strong precompactness of the sequences in $L^2([0,T]\times\dom)$. To achieve this, we will need to adapt the Aubin-Lions-Simon lemma~\cite{Temam1977,Simon1987} to our situation. This requires a uniform bound  on the time derivative of $u_h$ or $\hu_h$. We prove this in the following lemma:
\begin{lemma}
	\label{lem:timederivativebounds}
	Under the assumption that $h^{s-1}\leq C \sqrt{\Delta t}$, the time derivative  of   $\hu_h$ satisfies uniformly in $h>0$:
	\begin{equation}
		\label{eq:timederu}
		 \{\partial_t \hu_h\}_{h>0}\subset L^2(0,T;(H_{\Div}^1(\dom)\cap H^s(\dom))^*).
	\end{equation}
	for any $s\in \N$, $2\leq s\leq k+1$, where $k$ is the degree of the polynomial space used in $\Uh$.
\end{lemma}
\begin{proof}
	Note that
	\begin{equation*}
 		\partial_t \hu_h =\frac{\tu^{m+2}_h-\tu_h^{m+1}}{\Delta t},\quad t\in (t^m,t^{m+1}).
 \end{equation*}
	Recall that $\PUh:L^2(\dom)^d\to \Uh$ is the $L^2$-orthogonal projection onto the finite element space.
	Then we have
	for $v\in L^2(0,T;H^1_{\Div}\cap H^s(\dom))$, using~\eqref{eq:step1fully} and~\eqref{eq:projection1}
	\begin{align*}
		\int_0^T(\partial_t \hu_h,v) dt& = \int_0^T(\partial_t \hu_h,v-\PUh v) dt + \int_0^T(\partial_t \hu_h,\PUh v) dt \\
		& = \underbrace{\int_0^T(\partial_t \hu_h,v-\PUh v) dt}_{\text{I}}
		-\underbrace{\int_{\Delta t}^{T+\Delta t} b(\bar{u}_h,\tu_h,\PUh v(t-\Delta t))dt}_{\text{II}}-\underbrace{\int_{\Delta t}^{T+\Delta t}\mu(\Grad \tu_h,\Grad \PUh v(t-\Delta t))dt}_{\text{III}}\\
		&\quad +\underbrace{\int_{0}^{T}(p_h,\Div (\PUh v-v)) dt}_{\text{IV}} +\underbrace{\int_{\Delta t}^{T+\Delta t}(f_h,\PUh v(t-\Delta t))dt}_{\text{V}} ,
	\end{align*}
	where we have used that $v$ is divergence free. We extend $v$ and $f$ by zero for $t>T$. We estimate each of the five terms. 
	Since $\partial_t \hu_h\in \Uh$, we have that
	\begin{equation*}
		(\partial_t \hu_h,v)= (\partial_t \hu_h,\PUh v),
	\end{equation*} 
	and hence the first term $\text{I}$ vanishes.
	
	For the second term, we have
	\begin{align*}
		|\text{II}|&\leq \left| \sum_{m=1}^{N} \int_{t^m}^{t^{m+1}}b(\tu_h^m,\tu_h^{m+1},\PUh v(t-\Delta t))dt\right|\\
		& \leq \sum_{m=1}^{N-1} \left(\norm{\tu^m_h}_{L^2} \norm{\Grad\tu_h^{m+1}}_{L^2}+\frac12\norm{\tu^{m+1}_h}_{L^2} \norm{\Div\tu_h^{m}}_{L^2}\right)\int_{t^m}^{t^{m+1}}\norm{\PUh v(t-\Delta t)}_{L^\infty}dt\\
		& \leq C\max_{k=1,\dots,N}\norm{\tu^k_h}_{L^2}\left(\Delta t \sum_{m=1}^{N}\norm{\Grad\tu_h^{m}}_{L^2}^2\right)^{\frac12}\norm{\PUh v}_{L^2(0,T;L^\infty(\dom))}\\
		&\leq C (E_h^0+ \norm{f}^2_{L^2([0,T]\times\dom)})	\norm{\PUh v}_{L^2(0,T;L^\infty(\dom))}\\
		&\leq C	\norm{v}_{L^2(0,T;L^\infty(\dom))}\\
		&\leq C	\norm{v}_{L^2(0,T;H^2(\dom))}
	\end{align*}
	where we used Morrey's inequality, the Sobolev embedding theorem and the stability of the $L^2$-projection operator, equation~\eqref{eq:Lr}, for the last inequalities. 
	We proceed to term III:
	\begin{align*}
		|\text{III}|& \leq \mu \left|\sum_{m=1}^{N }\int_{t^m}^{t^{m+1}}(\Grad \tu^{m+1}_h,\Grad\PUh v(t-\Delta t))dt  \right|\\
		& \leq\mu \left(\Delta t \sum_{m=1}^{N-1}\norm{\Grad \tu^{m+1}_h}_{L^2}^2\right)^{\frac12}\norm{\Grad\PUh v}_{L^2([0,T]\times\dom)}\\
		&\leq C (E_h^0 	+\norm{f}_{L^2([0,T]\times\dom)}^2)\norm{\Grad v}_{L^2([0,T]\times\dom)}\\
		&\leq C \norm{v}_{L^2(0,T;H^1(\dom))},	
	\end{align*}
	where we have used the stability of $\PUh$ on $H^1$,~\eqref{eq:H1}. 
 	For the fourth term, IV, we have,
	\begin{align}\label{eq:pressuretimecont}
		|\text{IV}|&\leq \norm{p_h}_{L^2([0,T]\times \dom)}\norm{\Div(\PUh v-v)}_{L^2([0,T]\times\dom)}\\
		& \leq \frac{C}{\sqrt{\Delta t}}\left(\sum_{m=0}^{N-1}\norm{\tu^{m+1}_h-u^{m+1}_h}_{L^2}^2\right)^{\frac12}\left(\int_0^T|v-\PUh v|_{H^1}^2 dt\right)^{\frac12}\notag\\
		&\leq \frac{C (E_h^0+\norm{f}_{L^2([0,T]\times\dom)}^2)^{1/2} h^{s-1}}{\sqrt{\Delta t}}\norm{v}_{L^2(0,T;H^s(\dom))},\notag
	\end{align}
	using~\eqref{eq:H1approx}.
	To deduce the estimate for the pressure, we have used the LBB stability condition,~\eqref{eq:LBB}, combined with~\eqref{eq:projection1} and the Poincar\'e inequality:
	\begin{align*}
		\norm{p_h^{m+1}}_{L^2}&\leq \frac{1}{\beta^*}\sup_{0\neq v\in\Uh}\frac{(\Div v,p^{m+1}_h)}{\norm{\Grad v}_{L^2}}\\
		& =  \frac{1}{\beta^*}\sup_{0\neq v\in\Uh}\frac{\left( v,\frac{u^{m+1}_h-\tu^{m+1}_h}{\Delta t}\right)}{\norm{\Grad v}_{L^2}}\\
		& \leq C\norm{\frac{u^{m+1}_h-\tu^{m+1}_h}{\Delta t}}_{L^2}.
	\end{align*}
	Finally, for the last term,  V, we have
	\begin{equation*}
		|\text{V}|\leq \left|\int_{\Delta t}^{T+\Delta t}(f_h, v(t-\Delta t))dt \right|  \leq \norm{f_h}_{L^2(0,T;L^2(\dom))}\norm{v}_{L^2(0,T ;L^2(\dom))}\leq C\norm{f}_{L^2(0,T;L^2(\dom))}\norm{v}_{L^2(0,T;L^2(\dom))}.
	\end{equation*}
	Combining the estimates I--V, we obtain
	\begin{equation*}
		\sup_{0\neq v\in L^2(0,T;H^1_{\Div}\cap H^s)}\left|\frac{(\partial_t \hu_h,v)}{\norm{v}_{L^2(0,T;H^1_{\Div}\cap H^s)}}\right|\leq C\left(1+\frac{h^{s-1}}{\sqrt{\Delta t}}\right),
	\end{equation*}
	which is uniformly bounded under the assumption that $h^{s-1}\leq C\sqrt{\Delta t}$. Thus $\partial_t \hu_h\subset L^2(0,T;(H^1_{\Div}(\dom)\cap H^s(\dom))^*)$ uniformly in $h>0$. 
	
\end{proof}
Next, we will show that the sequences  $\{\tu_h\}_{h>0}$, $\{\hu_h\}_{h>0}$, $\{u_h\}_{h>0}$, $\{\bar{u}_h\}_{h>0}$ have the same weak limit $\tu=\hu=u=\bar{u}$:
\begin{lemma}
	\label{lem:samelimits}
	Assume that  $\Delta t = o_{h\to 0}(1)$. Then, we have 
	$u=\bar{u}=\tu=\hu$ a.e. in $[0,T]\times\dom$ and in $L^2([0,T]\times\dom)$, and
	\begin{align}
		\label{eq:uubarconv}
		\lim_{h\to 0}\norm{u_h-\hu_h}_{L^2([0,T]\times\dom)} &= 0,\\
		\label{eq:utildeuhatconv}
		\lim_{h\to 0}\norm{\hu_h-\tu_h}_{L^2([0,T]\times\dom)}& = 0,\\
		\label{eq:utildeubarconv}
		\lim_{h\to 0}\norm{\bar{u}_h-\tu_h}_{L^2([0,T]\times\dom)}& = 0.
	\end{align}	
\end{lemma}
\begin{proof}
		We have
	\begin{align*}
		\norm{\bar{u}_h-\tu_h}_{L^2([0,T]\times\dom)}^2 & = \Delta t\sum_{m=0}^{N-1}\int_{\dom}\left|\tu^m_h-\tu^{m+1}_h\right|^2 dx \\
		& \leq 2 \Delta t\sum_{m=0}^{N-1}\int_{\dom}\left|u^m_h-\tu^{m+1}_h\right|^2 + \left|\tu_h^m-u_h^m\right|^2 dx \\
		&\leq C\Delta t (E_h^0+\norm{f}_{L^2([0,T]\times\dom)}^2),
	\end{align*}
	by the energy balance, Lemma~\ref{lem:discenergyestimate} and Remark~\ref{rem:fneq0}. This shows~\eqref{eq:utildeubarconv}. 
	Using this and that both $\bar{u}_h$ and $\tu_h$ converge weakly, it then follows easily that for any test function $\varphi\in L^2([0,T]\times\dom) $,
	\begin{multline*}
		\int_0^T\!\!\int_{\dom}\bar{u}\cdot \varphi\, dx dt = \lim_{h\to 0} \int_0^T\!\!\int_{\dom}\bar{u}_h\cdot \varphi\, dx dt\\
		= \lim_{h\to 0}\int_0^T\!\!\int_{\dom}(\bar{u}_h-\tu_h)\cdot \varphi\, dx dt+ \lim_{h\to 0}\int_0^T\!\!\int_{\dom}\tu_h\cdot \varphi \,dx dt =\int_0^T\!\!\int_{\dom}\tu\cdot \varphi\, dx dt,
	\end{multline*}
	hence $\tu=\bar{u}$ a.e. in $[0,T]\times\dom$ and in $L^2([0,T]\times\dom)$.
In the same way, we compute
	\begin{align*}
		\norm{\hu_h-u_h}_{L^2([0,T]\times\dom)}^2 & = \sum_{m=0}^{N-1}\int_{t^m}^{t^{m+1}}\!\!\int_{\dom}\left|\frac{t-t^m}{\Delta t}(u^{m+1}_h-\tu^{m+2}_h)+ \frac{t^{m+1}-t}{\Delta t}(u^m_h - \tu^{m+1}_h)\right|^2 dx dt\\
		&\leq 2\Delta t\sum_{m=0}^{N-1}\int_{\dom}\left|u^{m+1}_h-\tu^{m+2}_h\right|^2+\left|\tu^{m+1}_h-u^m_h\right|^2 dx \\
		&\leq C\Delta t (E_h^0+\norm{f}_{L^2([0,T]\times\dom)}^2).
	\end{align*}
	Thus as $\Delta t,h \to 0$, we obtain~\eqref{eq:uubarconv}, and in the same way as before, one shows that the limits $u$ and $\hu$ agree almost everywhere. Lastly, we have
	 	\begin{align*}
	 	\norm{\hu_h-\tu_h}_{L^2([0,T]\times\dom)}^2 & = \sum_{m=0}^{N-1}\int_{t^m}^{t^{m+1}}\!\!\int_{\dom}\left|\frac{t-t^m}{\Delta t}(\tu^{m+2}_h-\tu^{m+1}_h)\right|^2 dx dt\\
	 	&\leq \Delta t\sum_{m=0}^{N-1}\int_{\dom}\left|\tu^{m+2}_h-\tu^{m+1}_h\right|^2 dx \\
	 	&\leq 2\Delta t\sum_{m=0}^{N-1}\int_{\dom}\left|u^{m+1}_h-\tu^{m+2}_h\right|^2+\left|\tu^{m+1}_h-u^{m+1}_h\right|^2 dx \\
	 	&\leq C\Delta t (E_h^0+\norm{f}_{L^2([0,T]\times\dom)}^2),
	 \end{align*}
	 which implies~\eqref{eq:utildeuhatconv} and in the same way as before, that the limits $\tu$ and $\hu$ agree almost everywhere.
\end{proof}
As a next step, we show that the limit $u$ is weakly divergence free, i.e., $(u(t),\Grad q)=0$ for almost every $t\in [0,T]$ and $q\in H^1(\dom)$. Since $u\in L^2(0,T;H^1_0(\dom))$ this then implies after integration by parts (the boundary term vanishes) that $\Div u=0$ a.e.. This follows from the next lemma:
\begin{lemma}\label{lem:divfree2}
	Any weak limit $u$ of the sequence $\{u_h\}_{h>0}$ satisfies for a.e. $t\in [0,T]$,
	\begin{equation*}
		(u,\Grad q) = 0,\quad \forall \, q\in H^1(\dom).
	\end{equation*}
	
\end{lemma}
\begin{proof}
	Since $u$ is the weak limit of $\{{u}_h\}_{h>0}$, we have (up to a subsequence)
	\begin{equation*}
		(u,\Grad q) = \lim_{h\to 0} ({u}_h,\Grad q),
	\end{equation*}
	for any $q\in H^1(\dom)$ for a.e. $t\in [0,T]$. Then since ${u}_h(t)\in H^1_0(\dom)$ for any fixed $h>0$ and $t\in [0,T]$, we can integrate by parts in the scheme~\eqref{eq:projection2} to obtain
	\begin{equation}\label{eq:weakapproxdivfree}
		0=(\Div {u}_h,q) = -({u}_h,\Grad q),
	\end{equation}
	for $q\in \Ph$.
	Then we can write
	\begin{equation*}
		|({u}_h ,\Grad q)| = 	|({u}_h ,\Grad (q-\PPh q)) |\leq \norm{{u}_h}_{L^2}\norm{\Grad (q-\PPh q)}_{L^2}\stackrel{h\to 0}{\longrightarrow} 0,
	\end{equation*}
	since ${u}_h\in L^\infty([0,T];L^2(\dom))$ uniformly in $h>0$ and $\PPh q \to q$ in $H^1(\dom)$ as $h\to 0$ for any $q\in H^1(\dom)$ by~\eqref{eq:L2projproperties}. Thus, in the limit $(u,\Grad q) = 0$ for any $q\in H^1(\dom)$ and a.e. $t\in [0,T]$.
\end{proof}

Next, we would like to apply the Aubin-Lions-Simon lemma~\cite{Simon1987} and obtain a strongly convergent subsequence of $\{u_h\}_{h>0}$ in $L^2([0,T]\times\dom)$. However,   we have the time continuity for $\tu_h$ with values in the dual space of $H^1_{\Div}(\dom)\cap H^s(\dom)$ and the spatial regularity in $H^1_0(\dom)$ for the sequence of approximations $\{\tu_h\}_{h>0}$, but they are not divergence free. In order to apply the Aubin-Lions lemma, we would either need them to be divergence free or uniform bounds $\partial_t \tu_h \in L^2(0,T;(H^1_0(\dom)\cap H^s(\dom))^*)$ for some $s\in \N$. However, this does not seem possible to achieve due to the lack of pressure bounds.
We therefore need a variant of a classical lemma by Lions~\cite[Ch.III, Lem. 2.1]{Temam1977}, also called Ehrling's lemma, adapted to our setting. 
The lemma we need is inspired by a similar result in~\cite{Eymard2024}. We first define the space
\begin{equation}
	\label{eq:spaceV}
	\Vh = \left\{v\in L^2(\dom)\, ,\, (v,\Grad q) = 0,\quad \forall q\in \Ph \right\}.
\end{equation}
Note that this is not a finite element space and not finite dimensional. We clearly have $L^2_{\Div}(\dom)\subset \Vh$. We denote the orthogonal projection from $L^2(\dom)$ onto $\Vh$ by $\PVh:L^2(\dom)\to \Vh$, i.e., we have for any $v \in L^2(\dom)$ that
\begin{equation*}
	(\PVh v,w) = (v,w),\quad \forall w\in \Vh.
\end{equation*}
Moreover, $\PVh$ is the identity on $\Vh$.
This implies that
\begin{equation*}
	\norm{\PVh v - v}_{L^2(\dom)}^2 = \norm{\PVh v }_{L^2(\dom)}^2 + \norm{v}_{L^2(\dom)}^2 - 2(\PVh v,v) = \norm{v}_{L^2(\dom)}^2 -  \norm{\PVh v }_{L^2(\dom)}^2,
\end{equation*}
thus
\begin{equation*}
	\norm{\PVh v}_{L^2(\dom)}\leq \norm{v}_{L^2(\dom)}. 
\end{equation*}
We note that due to~\eqref{eq:weakapproxdivfree}, we have $\PVh \tu^{m+1}_h = u^{m+1}_h$ for all $m$ and hence $u_h = \PVh \hu_h(\cdot -\Delta t)$.

We also have the following simple `commuting' property for $\PVh$ and $\Pl: L^2(\dom)\to L^2_{\Div}(\dom)$, the Leray projector:
\begin{lemma}
	\label{lem:commuting}
	We have that
	\begin{equation}\label{eq:commuting}
		\Pl \PVh v = \PVh \Pl v = \Pl v,\quad \forall \, v\in L^2(\dom).
	\end{equation}
\end{lemma}
\begin{proof}
	Since $L^2_{\Div}(\dom)\subset \Vh$ for any $h>0$, we have that $\PVh$ is the identity on $L^2_{\Div}(\dom)$ and hence $\PVh \Pl v = \Pl v$. For the other identity, we take $v,w\in L^2(\dom)$ and compute, using that $\Pl$ and $\PVh$ are self-adjoint,
	\begin{equation*}
		(\Pl\PVh v,w) = (\PVh v, \Pl w) = (v,\PVh\Pl w) = (v,\Pl w) = (\Pl v,w),
	\end{equation*}
	implying the claim.
\end{proof}
Using this, we will prove:
\begin{lemma}\label{lem:lionslike}
	Denote $Z = (H^k(\dom)\cap H^1_{\Div}(\dom))^*$ for some $k\in \N$. Then for any $\eta>0$ there exists a constant $C_\eta>0$ and $h_\eta>0 $ such that for all $h\leq h_{\eta}$ we have
	\begin{equation}\label{eq:lionslike}
		\norm{\PVh w}_{L^2(\dom)}\leq \eta \norm{w}_{H^1_0(\dom)} + C_\eta \norm{w}_Z,\quad \forall w\in H^1_0(\dom).
	\end{equation}
\end{lemma}	
\begin{proof}
	The proof of this lemma is by contradiction. Assume that~\eqref{eq:lionslike} does not hold. Then there exists $\eta>0$ such that for all $C>0$ and $\bar{h}>0$ there exists $0<h<\bar{h}$ and $v\in H^1_0(\dom)$ such that
	\begin{equation*}
		\norm{\PVh v}_{L^2(\dom)}> \eta \norm{v}_{H^1_0(\dom)} + C \norm{v}_Z.
	\end{equation*}
	We generate a sequence $\{v_N\}_N$ by taking $C=\bar{h}^{-1} =n\in \N$ and $h^{-1}:=N\geq n$ strictly increasing (so that $h$ is strictly decreasing), such that
	\begin{equation*}
		\norm{\PN v_N}_{L^2(\dom)}> \eta \norm{v_N}_{H^1_0(\dom)} + n \norm{v_N}_Z,
	\end{equation*}
	where we denoted $\PN:=\PVh = \mathcal{P}_{\mathbb{V}_{N^{-1}}}$.
	We define the mapping $N = \tau(n)$ which is strictly increasing and hence has a strictly increasing inverse $\tau^{-1}$. Then the last statement can be rewritten in terms of $N$ as
	\begin{equation*}
		\norm{\PN v_N}_{L^2(\dom)}> \eta \norm{v_N}_{H^1_0(\dom)} + \tau^{-1}(N)\norm{v_N}_Z.
	\end{equation*}
	Next, we define $w_N =v_N / \norm{\PN v_N}_{L^2(\dom)}$. Clearly $\norm{\PN w_N}_{L^2(\dom)} = 1$. Therefore, the last statement translates to
	\begin{equation*}
		1>\eta \norm{w_N}_{H^1_0} + \tau^{-1}(N)\norm{w_N}_{Z}.
	\end{equation*}
	This implies that $\norm{w_N}_{H^1_0}$ is uniformly bounded in $N$ and therefore $\{w_N\}_{N}$ has a subsequence that converges strongly in $L^2(\dom)$ to a limit $w\in H^1_0(\dom)$, due to the compact embedding of $H^1_0$ into $L^2$. We still denote the subsequence by $N$ for convenience of notation. Then we must have that $\norm{w_N}_Z\to 0$ as $N\to \infty$ since $\tau^{-1}(N)\to \infty$ and the left hand side is bounded. Hence $\norm{w}_Z = 0$. 
	Since 
	\begin{equation*}
		\begin{split}
			\norm{w}_Z &= \sup_{\varphi \in H^k\cap H^1_{\Div}(\dom)} \frac{|(w,\varphi)|}{\norm{\varphi}_{H^k\cap H_{\Div}^1}} \\
			&= \sup_{\varphi \in H^k\cap H^1_{0}(\dom)} \frac{|(w,\Pl\varphi)|}{\norm{\Pl\varphi}_{H^k\cap H_0^1}}\\
			& = \sup_{\varphi \in H^k\cap H^1_{0}(\dom)} \frac{|(\Pl w,\varphi)|}{\norm{\Pl\varphi}_{H^k\cap H^1_0}} \\
			&\geq \sup_{\varphi \in H^k\cap H^1_{0}(\dom)} \frac{|(\Pl w,\varphi)|}{\norm{ \varphi}_{H^k\cap H^1_0}} \\
			&= \norm{\Pl w}_{(H^k\cap H^1_0)^*}
		\end{split}
	\end{equation*}
	using the properties of the Leray-projector $\Pl$, we obtain that
	\begin{equation*}
		\norm{\Pl w}_{(H^k\cap H^1_0)^*}=0, 
	\end{equation*}
	i.e., $\Pl w =0$ in $(H^k\cap H^1_0)^*$. That means that
	\begin{equation*}
		(\Pl w,\varphi) =0,\quad \forall \, \varphi\in H^k(\dom)\cap H^1_0(\dom).
	\end{equation*}
	Since $\Pl w\in L^2(\dom)$ and $H^1_0(\dom)\cap H^k(\dom)$ is dense in $L^2(\dom)$, we can extend this identity to $\varphi \in L^2(\dom)$ and conclude that $\norm{\Pl w}_{L^2(\dom)} = 0$. Next, we want to show that $\PN w_N\to \Pl w$ in $L^2(\dom)$ which would then imply a contradiction, since $\norm{\PN w_N}_{L^2}=1\neq 0=\norm{\Pl w}_{L^2}$. We start by showing that $\PN w_N \weak \Pl w$ in $L^2(\dom)$. Let us take $\varphi\in L^2(\dom)$ and decompose using the Helmholtz-decomposition as
	$\varphi =\Pl \varphi + \Grad g$ where $g\in H^1(\dom)$. Then we have, by Lemma~\ref{lem:commuting}, and using that $w_N\to w$ in $L^2(\dom)$, and that $\Pl$ is self-adjoint,
	\begin{equation*}
		(\PN w_N, \Pl \varphi ) = (w_N,\PN \Pl \varphi) = (w_N,\Pl \varphi) \stackrel{N\to \infty}{\longrightarrow} (w,\Pl\varphi) = (\Pl w,\varphi).
	\end{equation*}
	On the other hand, we have, using that $\Grad \PPh g \to \Grad g$ in $L^2(\dom)$ as $h\to 0$, and that $(\PN w_N,\Grad q) = 0$ for all $q\in \mathbb{P}_{N^{-1}}$
	\begin{equation*}
		|	(\PN w_N,\Grad g)| =| (\PN w_N,\Grad (g- \mathcal{P}_{\mathbb{P}_{N^{-1}}}g))| \leq \norm{\PN w_N}_{L^2}\norm{\Grad (g- \mathcal{P}_{\mathbb{P}_{N^{-1}}}g)}_{L^2} \stackrel{N\to \infty}{\longrightarrow} 0.
	\end{equation*}
	Combining these two insights, we obtain
	\begin{equation*}
		(\PN w_N,\varphi)\stackrel{N\to \infty}{\longrightarrow} (\Pl w,\varphi),
	\end{equation*}
	thus $\PN w_N\weak \Pl w$ in $L^2(\dom)$. Next, we consider the norm of $\PN w_N$:
	\begin{equation*}
		\norm{\PN w_N}_{L^2}^2 = (\PN w_N,\PN w_N) = (\PN w_N,w_N) \stackrel{N\to \infty}{\longrightarrow} (\Pl w,w) = \norm{\Pl w}_{L^2}^2,
	\end{equation*}
	where we used that $w_N\to w$ strongly and $\PN w_N\weak \Pl w$ weakly in $L^2$ and the properties of the orthogonal projections $\Pl$ and $\PN$. This implies that $\PN w_N\to \Pl w$ in $L^2$ as follows:
	\begin{equation*}
		\norm{\PN w_N - \Pl w}_{L^2}^2 = \norm{\PN w_N}_{L^2}^2 + \norm{\Pl w}_{L^2}^2 - 2(\PN w_N, \Pl w)\stackrel{N\to \infty}{\longrightarrow} 0,
	\end{equation*}
	using the weak convergence of $\PN w_N$ and the convergence of the $L^2$-norm. Thus, we conclude that 
	\begin{equation*}
		1=\lim_{N\to \infty}	\norm{\PN w_N}_{L^2}^2 =\norm{\Pl w}_{L^2}^2=0,
	\end{equation*}
	which yields the desired contradiction.
\end{proof}

With this technical lemma at hand, we can now prove the following modified Aubin-Lions lemma:

\begin{lemma}
	\label{lem:aubinlionsmodified}
	Let  $\{\hu_h\}_{h>0}, \{u_h\}_{h>0}$ be  sequences of functions $\hu_h,u_h: [0,T]\times\dom\to \R^d$ satisfying
	\begin{equation}\label{eq:samelimit}
		\lim_{h\to 0}\norm{ u_h-\hu_h}_{L^2([0,T]\times \dom)}=0,	
		\end{equation}
	and
	\begin{equation}
		\label{eq:aprioriAL}\begin{split}
			& u_h\in L^2(0,T;L^2(\dom)\cap \Vh),\\
			&\hu_h\in  L^2(0,T;H^1_0(\dom)),\\
			&\partial_t \hu_h\in L^2(0,T;(H^k(\dom)\cap H^1_{\Div}(\dom))^*),	
		\end{split}
	\end{equation}
	uniformly in $h>0$ where $k\in \N$. Then the sequences $\{\hu_h\}_{h>0}$ and $\{u_h\}_{h>0}$ are precompact in $L^2([0,T]\times\dom)$.
\end{lemma}
\begin{proof}
	The proof follows the proof of the Aubin-Lions-Simon lemma with modifications to account for the fact that $\partial_t \hu_h\in (H^k(\dom)\cap H^1_{\Div}(\dom))^*$ but $\hu_h$ is not divergence free.
	
	From the Banach-Alaoglu theorem, we obtain that $u_h\weak u$ and $\hu_h\weak \hu$ in $L^2([0,T]\times\dom)$ up to a subsequence, for the ease of notation still denoted by $h>0$. From Lemma~\ref{lem:divfree2}, we obtain that the limit $u$ is divergence free (and $\PVh u =u$), and from Lemma~\ref{lem:samelimits}, we obtain that $u=\hu\in L^2(0,T;L^2_{\Div}(\dom))\cap L^2(0,T;H^1_0(\dom))$ almost everywhere and in $L^2([0,T]\times\dom)$. We let $v_h:=\hu_h-u$ and $w_h=u_h-u$, then $v_h$ and $w_h$ satisfy~\eqref{eq:samelimit} and~\eqref{eq:aprioriAL} and converge weakly to $0$ in $L^2([0,T]\times\dom)$. Note also that $w_h = \PVh w_h$.
	Let us denote $Z =  (H^k(\dom)\cap H^1_{\Div}(\dom))^*$. We start by showing that for a.e. $t\in [0,T]$,
		\begin{equation}\label{eq:uniformZ}
		\norm{v_h(t)}_{Z}\leq C,
	\end{equation}
	and that
	\begin{equation}\label{eq:aeliminZ}
		\norm{v_h(t)}_Z \stackrel{h\to 0}{\longrightarrow} 0,\quad \text{a.e. }\, t\in [0,T].
	\end{equation}
	Then, by the Lebesgue dominated convergence theorem, 
	we can conclude that $v_n\to 0$ in $L^2(0,T;Z)$. Without loss of generality, we let $t=0$. Then for any $s,t>0$,
	\begin{equation*}
		\begin{split}
			v_h(0)&= v_h(t)-\int_0^t v'_h(\tau)d\tau\\
			& = \frac{1}{s}\left(\int_0^s v_h(t)dt -\int_0^s \int_0^t v_h'(\tau)d\tau dt\right)\\
			&:= a_h + b_h.	
	\end{split}\end{equation*}
	We can rewrite $b_h$ for all $h>0$ as
	\begin{equation*}
		b_h = -\frac{1}{s}\int_0^s (s-t) v_h'(t)dt.
	\end{equation*}
	Thus given any $\epsilon>0$, we can choose $s>0$ such that for all $n\in\N$ it holds
	\begin{equation*}
		\norm{b_h}_{Z}\leq \int_0^s \norm{v_h'(t)}_{Z} dt\leq \sqrt{s}\norm{v_h'}_{L^2(0,T;Z)} \leq \frac{\epsilon}{2},
	\end{equation*}
	since $v_h'\in L^2(0,T;Z)$ uniformly in $h>0$ by the assumptions.
	On the other hand, we have, since $Z\subset (H^1_0(\dom)\cap H^k(\dom))^*$ that
	\begin{equation*}
		\norm{a_h}_{Z}\leq C \norm{a_h}_{(H^1_0(\dom)\cap H^k(\dom))^*} \leq  C \norm{a_h}_{L^2(\dom)}= C \norm{\frac{1}{s}\int_0^s v_h(t)dt}_{L^2(\dom)}.
	\end{equation*}
	Due to the compact embedding $H^1_0(\dom)\subset\subset L^2(\dom)$ and~\eqref{eq:aprioriAL}, we have that $\frac{1}{s}\int_0^s v_h(t)dt\in H^1(\dom)$ is precompact in $L^2(\dom)$ and converges up to subsequence to zero in $L^2(\dom)$.
	Thus, we can find $\bar{h}>0$ such that for $h<\bar{h}$ (along a subsequence),
	\begin{equation*}
		\norm{a_h}_{Z}<\frac{\epsilon}{2}.
	\end{equation*}
	Thus~\eqref{eq:aeliminZ} holds.
	The uniform bound~\eqref{eq:uniformZ} follows in the same way, going over these estimates and observing that $a_h$ and $b_h$ are uniformly bounded independently of the time $t\in [0,T]$. 
	Hence, we obtain that $v_h(t)\to 0$ almost everywhere in $[0,T]$ and hence in $L^2(0,T;Z)$.
	Next, we have by Lemma~\ref{lem:lionslike},  that for any $\eta>0$, there exists $C_{\eta}>0$ and $h_\eta>0$ such that for all $h<h_\eta$,
	\begin{equation*}
		\norm{\PVh v_h}_{L^2([0,T]\times\dom)}\leq \eta \norm{v_h}_{L^2(0,T;H^1_0(\dom))}+C_{\eta}\norm{v_h}_{L^2(0,T;Z)}.
	\end{equation*}
	By the assumption~\eqref{eq:aprioriAL}, we have that
	\begin{equation*}
		\norm{v_h}_{L^2(0,T;H^1_0(\dom))}\leq C_0,
	\end{equation*}
	for some constant $C_0>0$ independent of $h$. 
	We now let $\epsilon>0$ arbitrary and pick $\eta = \epsilon/(2 C_0)$ (and $h_\eta$ correspondingly). As we have already shown that $\norm{v_h}_{L^2(0,T;Z)}\to 0$, we can choose $0<h_1\leq h_\eta$ such that for all $h<h_1$, we have
	\begin{equation*}
		C_\eta\norm{v_h}_{L^2(0,T;Z)}\leq \frac{\epsilon}{2}.
	\end{equation*}
	Thus for $h<h_1$, we have
	\begin{equation}\label{eq:PVhvconv}
		\norm{\PVh v_h}_{L^2([0,T]\times\dom)}\leq \epsilon.
	\end{equation}
	So we obtain that $\PVh v_h \to 0$ strongly in $L^2([0,T]\times\dom)$ along a subsequence. Next, we write
	\begin{equation*}
		v_h = v_h - \PVh v_h -w_h +\PVh w_h +\PVh v_h,
	\end{equation*}
	using that $w_h = \PVh w_h$. So we can estimate using the triangle inequality,
	\begin{align*}
		\norm{v_h}_{L^2([0,T]\times\dom)}&\leq \norm{v_h-w_h}_{L^2([0,T]\times\dom)} + \norm{\PVh v_h-\PVh w_h}_{L^2([0,T]\times\dom)} + \norm{\PVh v_h}_{L^2([0,T]\times\dom)}\\
		&\leq 2\norm{v_h-w_h}_{L^2([0,T]\times\dom)}   + \norm{\PVh v_h}_{L^2([0,T]\times\dom)},
	\end{align*}
	using that $\PVh$ is an orthogonal projection. The first term on the right hand side goes to zero by assumption~\eqref{eq:samelimit} and the second goes to zero thanks to~\eqref{eq:PVhvconv}. Thus $v_h\to 0$ in $L^2([0,T]\times\dom)$ up to a subsequence and by~\eqref{eq:samelimit} again, also $w_h\to 0$ in $L^2([0,T]\times\dom)$.
	Thus $v_h, \PVh v_h,w_h$ are precompact in $L^2([0,T]\times\dom)$ implying the precompactness of the sequences $\{\hu_h\}_{h>0}$, $\{u_h\}_{h>0}$ and $\{\PVh \hu_h\}_{h>0}$ in $L^2([0,T]\times\dom)$.
\end{proof}

Clearly, the discrete energy inequality, Lemma~\ref{lem:discenergyestimate}, the time-continuity result, Lemma~\ref{lem:timederivativebounds}, Lemma~\ref{lem:samelimits}, and Lemma~\ref{lem:divfree2}, enable us to apply Lemma~\ref{lem:aubinlionsmodified} to obtain that $\{u_h\}_{h>0}$, $\{\widetilde{u}_h\}_{h>0}$, $\{\widehat{u}_h\}_{h>0}$  and $\{\bar{u}_h\}_{h>0}$ defined in~\eqref{eq:defuh}--\eqref{eq:defph} converge up to a subsequence strongly in $L^2([0,T]\times\dom)$ to a limit $u\in L^\infty(0,T;L^2_{\Div}(\dom))\cap L^2(0,T;H^1_0(\dom))$.  

Using these preliminary estimates, we can now prove:
\begin{theorem}\label{thm:convergence}
	The sequences $\{u_h,\bar{u}_h,\hu_h,\widetilde{u}_h \}_{h>0}$ converge up to a subsequence to a Leray-Hopf solution $u$ of~\eqref{eq:NS} as in Definition~\ref{def:weaksol} under the condition that {$h^k = o(\sqrt{\Delta t})$} as $h,\Delta t\to 0$, where $k$ is the polynomial degree of the finite element space $\Uh$.
\end{theorem}
\begin{proof}
	Using the definitions of the interpolations~\eqref{eq:defuh}--\eqref{eq:defph}, we can rewrite the numerical scheme~\eqref{eq:step1fully}--\eqref{eq:projectionfullydiscrete} as
		\begin{equation}
			\label{eq:udiscinterp}
			\left(\partial_t u_h,v\right)+ b(\bar{u}_h,\tu_h,v)+\mu(\Grad \tu_h,\Grad v) = (p_h,\Div v)+(f_h,v),
		\end{equation}
			and 
	\begin{equation}
		\label{eq:divconstraintinterp}
		(\Div u_h,q)=0,
	\end{equation}
	with the test functions in the same space as in~\eqref{eq:step1fully}--\eqref{eq:projectionfullydiscrete}.
	By the previous considerations, weak convergences~\eqref{eq:uhweakconv}--\eqref{eq:Hhweakconv}, Lemma~\ref{lem:discenergyestimate}, Lemma~\ref{lem:timederivativebounds}, Lemma~\ref{lem:samelimits}, and Lemma~\ref{lem:aubinlionsmodified}, we can improve~\eqref{eq:uhweakconv}--\eqref{eq:Hhweakconv} to
	\begin{align}
		u_h,\bar{u}_h,\hu_h,\widetilde{u}_h\to u,&\quad \text{in }\, L^2(0,T;L^2(\dom)),\label{eq:uhstrongconv}\\
		\bar{u}_h,\hu_h,\widetilde{u}_h\weak u,&\quad \text{in }\, L^2(0,T;H^1_0(\dom)).\label{eq:tuhweakconvagain}
	\end{align}
	From Lemma~\ref{lem:divfree2}, it follows that the limit $u$ is divergence free almost everywhere.

	Now we take a test function $v \in C^\infty_c((0,T)\times \dom)$ that is divergence free, and $\Delta t$ small enough such that $v(t,\cdot)=0$ on $[0,2\Delta t]$ and $[T-2\Delta t, T]$.
	We integrate in time and rewrite~\eqref{eq:udiscinterp} as 
		\begin{align}
			\label{eq:uapp}
			&\int_0^T\Big[\underbrace{\left(\partial_t u_h,v\right)}_{(i)}+ \underbrace{b(\bar{u}_h,\tu_h,v)}_{(ii)}+\underbrace{\mu(\Grad \tu_h,\Grad v)}_{(iii)} -\underbrace{(f_h,v)}_{(iv)}\Big] dt\\
			&=\int_0^T\Big[\underbrace{\left(\partial_t u_h,v-\PUh v\right)}_{(v)}+ \underbrace{b(\bar{u}_h,\tu_h,v-\PUh v)}_{(vi)}+\underbrace{\mu(\Grad \tu_h,\Grad (v-\PUh v))}_{(vii)}\Big] dt\notag\\
			&\quad +\int_0^T\Big[\underbrace{(p_h,\Div \PUh v)}_{(viii)}-\underbrace{(f_h,v-\PUh v)}_{(ix)}\Big] dt,\notag
		\end{align}
	where $\PUh$ is the $L^2$-orthogonal projection onto $\Uh$ as before. We consider the limits $h,\Delta t\to 0$ in each of these terms. For the first term, we integrate by parts to obtain
	\begin{equation*}
		\int_0^T(i) dt = -\int_0^T(u_h,\partial_t v) dt \stackrel{h\to 0}{\longrightarrow}-\int_0^T(u ,\partial_t v) dt ,
	\end{equation*}
	using that $u_h\weak u$ in $L^2([0,T]\times\dom)$ . For the second term, we use that $\bar{u}_h,\tu_h\to u$ in $L^2([0,T]\times\dom)$ and $\bar{u}_h,\tu_h\weak u$ in $L^2([0,T];H^1_0(\dom))$, and that $u$ is almost everywhere divergence free,
	\begin{align*}
		\int_0^T (ii) dt& = \int_0^T \left((\bar{u}_h\cdot\Grad)\tu_h\cdot v +\frac12 \Div\bar{u}_h \tu_h\cdot v\right) dt\\
&	\stackrel{h\to 0}{\longrightarrow} \int_0^T\int_{\dom}\left((u\cdot\Grad) u \cdot v +\frac12\Div u u \cdot v \right) dx dt =  \int_0^T   (u\cdot\Grad) u \cdot v  dt.
	\end{align*}
	For the third term, we use that $\Grad\tu_h\weak \Grad u$ in $L^2([0,T]\times\dom)$ and hence
	\begin{equation*}
		\int_0^T (iii) dt \stackrel{h\to 0}{\longrightarrow}\mu \int_0^T (\Grad u,\Grad v) dt.
	\end{equation*}
	We consider the fourth term: We have
		\begin{equation*}
			\begin{split}
					\int_0^T (iv) dt &= \sum_{m=0}^{N-1}\int_{t^m}^{t^{m+1}} (\PUh f^m,v) dt\\
					& = \sum_{m=0}^{N-1}\int_{t^m}^{t^{m+1}} (  f^m,\PUh v) dt\\
					& = \sum_{m=0}^{N-1}\int_{t^m}^{t^{m+1}} \frac{1}{\Delta t}\int_{t^m}^{t^{m+1}}(  f(s),\PUh v-v)ds dt + \sum_{m=0}^{N-1}\int_{t^m}^{t^{m+1}} \frac{1}{\Delta t}\int_{t^m}^{t^{m+1}}(  f(s), v) ds dt\\
					& \stackrel{h,\Delta t\to 0}{\longrightarrow} \int_0^T (f,v) dt,
			\end{split}
	\end{equation*}
	since time averages converge in $L^2$ and $\norm{v-\PUh v}_{L^2(\dom)}\to 0$ as $h\to 0$. Hence the left hand side of~\eqref{eq:uapp} converges to
	\begin{equation*}
		\int_0^T\left[-( u ,\partial_t v)+ b(u,u,v)+\mu(\Grad u,\Grad v) -(f,v)\right] dt .
	\end{equation*}
	In order to conclude that $u$ satisfies the weak formulation of~\eqref{eq:NS}, we thus need to show that the right hand side converges to zero.
	 So, we consider the terms on the right hand side: For term (v), since $\partial_t u_h\in \Uh$, we have $(\partial_t u_h,v)=(\partial_t u_h,\PUh v)$, hence this term vanishes.
	Next,
	\begin{equation*}
		\begin{split}
			\left|\int_0^T (vi)dt \right|&=\left|\int_0^Tb(\bar{u}_h,\tu_h,v-\PUh v) dt\right|\\
			&\leq \left(\norm{\bar{u}_h}_{L^\infty(0,T;L^2)}\norm{\Grad\tu_h}_{L^2([0,T]\times\dom)} +\frac12\norm{\tu_h}_{L^\infty(0,T;L^2)}\norm{\Div \bar{u}_h}_{L^2([\Delta t,T]\times\dom)}\right)\norm{v-\PUh v}_{L^2([0,T]\times\dom)}\\
			&\leq C h \norm{v }_{L^2(0,T;H^1(\dom))},
		\end{split}
	\end{equation*}
	where we used the approximation property of the $L^2$-projection,~\eqref{eq:L2approx}, and the energy estimate, Lemma~\ref{lem:discenergyestimate}. Thus also this term vanishes as $h,\Delta t \to 0$.
	For term (vii), we have
	\begin{equation*}
		\begin{split}
				\left|\int_0^T (vii)dt \right|&=	\left|\mu\int_0^T(\Grad\tu_h,\Grad(v-\PUh v))dt\right|\\
				&\leq \mu \norm{\Grad\tu_h}_{L^2([0,T]\times\dom)}\norm{\Grad(v-\PUh v)}_{L^2([0,T]\times\dom)}\\
				&\leq C h\norm{v}_{L^2(0,T;H^2(\dom))},
		\end{split}
	\end{equation*}
	again using the approximation property of the $L^2$-projection,~\eqref{eq:H1approx}, and the discrete energy estimate. 
	For the term (viii), we have, using that $v$ is divergence free,
	\begin{equation}\label{eq:pressuretermestimate}
		\begin{split}
		\left|\int_0^T (viii) dt\right|& =	\left|\int_0^T(p_h,\Div(\PUh v-v)) dt\right|\\
		&\leq  \norm{p_h}_{L^2([0,T]\times \dom)}\norm{\Div(\PUh v-v)}_{L^2([0,T]\times\dom)}\\
			& \leq \frac{1}{\sqrt{\Delta t}}\left(\sum_{m=0}^{N-1}\norm{\tu^{m+1}_h-u^{m+1}_h}_{L^2}^2\right)^{\frac12}\left(\int_0^T|v-\PUh v|_{H^1}^2 dt\right)^{\frac12} \\
			&\leq \frac{C (E_h^0+\norm{f}_{L^2([0,T]\times\dom)}^2)^{1/2} h^k}{\sqrt{\Delta t}}\norm{v}_{L^2(0,T;H^{k+1}(\dom))}, 
		\end{split}
	\end{equation}
	similar to the estimate~\eqref{eq:pressuretimecont}, and where we have used~\eqref{eq:H1approx}.  Clearly, also this term goes to zero if $h^k=o(\sqrt{\Delta t})$. 
	Finally, for the ninth term, we have, since $f_h=\PUh f^m$,
	\begin{equation*}
		(ix) = (f_h,v-\PUh v) = 0,
	\end{equation*}
	i.e., this term vanishes.
	Thus we obtain in the limit
	\begin{equation}
		\label{eq:uprelimlimit}
		\int_0^T\left[-( u ,\partial_t v)+ b(u,u,v)+\mu(\Grad u,\Grad v) -(f,v)\right] dt = 0,
	\end{equation}
	for any divergence free $v\in C^\infty_c((0,T)\times\dom))$. Passing to the limit in the discrete energy balance~\eqref{eq:discenergybalance} and using the properties of weak convergence, we obtain that the limit $u$ satisfies for almost any $t\in [0,T]$,
	\begin{equation}\label{eq:energylimit}
	\frac{1}{2}\norm{u(t)}_{L^2(\dom)}^2 +\mu\int_0^t\norm{\Grad u(s)}_{L^2(\dom)}^2 ds \leq \frac12\norm{u_0}_{L^2(\dom)}^2 + \int_0^t(f(s),u(s)) ds,
	\end{equation}
	which is the energy inequality~\eqref{eq:energyineq}. This implies the spatial regularity required in~\eqref{eq:regularity}. To see that $u$ satisfies the time regularity in~\eqref{eq:regularity}, we note that due to the spatial regularity of $u$, we can use the density of smooth functions and extend the weak formulation~\eqref{eq:uprelimlimit} to functions $v\in L^4(0,T;H^1_{\Div}(\dom))$. Finally, we establish the continuity at zero. We write
	\begin{equation*}
		\norm{u(t)-u_0}_{L^2(\dom)}^2 = \norm{u(t)}_{L^2(\dom)}^2 + \norm{u_0}_{L^2(\dom)}^2 - 2 (u(t),u_0).
	\end{equation*}
First, we note that $\partial_t u\in L^{4/3}(0,T;(H^1_{\Div}(\dom)^*))$ and $u\in L^\infty(0,T;L^2_{\Div}(\dom))$ imply that $u$ is weakly continuous in time with values in $L^2_{\Div}(\dom)$, i.e., the mapping
\begin{equation*}
	t\mapsto (u(t),v),
\end{equation*}
is continuous for any $v\in L^2_{\Div}(\dom)$ (see for example Lemma II.5.9 in~\cite{Boyer2013} or~\cite{Temam1977}). This implies that
\begin{equation*}
	\lim_{t\to 0} (u(t),u_0) = \norm{u_0}_{L^2(\dom)}^2.
\end{equation*}
Furthermore, from Fatou's lemma, we obtain that
\begin{equation*}
	\norm{u_0}_{L^2(\dom)}^2 \leq \lim_{t\to 0}\norm{u(t)}_{L^2(\dom)}^2. 
\end{equation*}
On the other hand, sending $t\to 0$ in the energy inequality~\eqref{eq:energylimit}, we obtain that
\begin{equation*}
	\lim_{t\to 0}\frac12\norm{u(t)}_{L^2(\dom)}^2 \leq \frac12\norm{u_0}_{L^2(\dom)}^2.
\end{equation*}
Thus
\begin{equation*}
	\lim_{t\to 0} \norm{u(t)}_{L^2(\dom)}^2 = \norm{u_0}_{L^2(\dom)}^2 
\end{equation*}
and we conclude that
	\begin{equation*}
	\lim_{t\to 0}\norm{u(t)-u_0}_{L^2(\dom)}^2 = 0.
\end{equation*}
	Therefore,   $u$ is a Leray-Hopf weak solution of~\eqref{eq:NS}. 
\end{proof}
\section{Convergence of projection method with projection step as a Poisson problem}\label{sec:poisson}
While the discretization of the projection step in~\eqref{eq:projection} could be considered as a Darcy problem, one can also consider a discretization as a Poisson problem. In this case, one can obtain stability and convergence of a fully discrete scheme for the incremental version of the projection method. This was already done in~\cite{Eymard2024} by the use of a special interpolation operator which was constructed explicitly for the lowest degree Taylor-Hood finite elements. No inf-sup condition such as~\eqref{eq:LBB} needed to be assumed for this to work. We will show here that the proof can be modified to prove that the linearly implicit discretization of the incremental projection method below converges without the need for this interpolation operator and therefore independently of the degree of the finite elements. This requires an adaption of the compactness arguments in~\cite{Eymard2024} but also shortens the proof significantly. We will outline the compactness arguments here and sketch the rest of the proof as it can be achieved in a similar way to the proof of convergence of Chorin's version of the projection method in the previous sections or alternatively the proof in~\cite{Eymard2024}.
We consider the following discretization of~\eqref{eq:NS}: We use the same notation and trilinear form $b$ as in Section~\ref{sec:spatial} with the following modifications. We will seek the intermediate velocity $\tu_h^{m+1}$ in a space $\Uh\subset H^1_0(\dom)$, the pressure $p_h^{m+1}$ in a space $\Ph\subset H^1(\dom)\cap L^2_0(\dom)$ and the final velocity $u_h^{m+1}$ in the space $\Yh = \Uh + \Grad \Ph$. Note that now $\Ph\subset H^1(\dom)$ and therefore $\Grad \Ph$ is well-defined (though possibly containing discontinuous functions). Also note that the final velocity $u_h^{m+1}$ may not satisfy homogeneous Dirichlet boundary conditions in this case.  We again require both the $L^2$-orthogonal projections $\PUh:L^2(\dom)\to \Uh$ and $\PPh:L^2(\dom)\to \Ph$ defined by~\eqref{eq:L2projection} to satisfy the properties~\eqref{eq:L2projproperties} for possibly different $k$. 

We start by approximating the initial data. Given $u_0\in L^2_{\Div}(\dom)$, we let $\tu_h^{0} = \PUh u_0$,  and define $u_h^0$ and $p^0_h$ by projecting it onto $\Vh\cap \Yh\times\Ph$, 
using the following projection step:
\begin{align*}
	\left(\frac{u^{0}_h-\tu^{0}_h}{\Delta t }, v\right)&= -(\Grad p^{0}_h, v),\\
	\left(u^{0}_h,\Grad q\right)&=0,	
\end{align*}
for all $(v,q)\in \Yh\times \Ph$.
Then we compute updates using the following scheme:
\begin{enumerate}
	\item[{\bf Step 1}] (Prediction step):
	For any $m\geq 0$, given $\tu_h^m\in \Uh$, $ u_h^{m} \in \Yh$, $p_h^m\in \Ph$, find $\tu_h^{m+1} \in \Uh $, such that for all $v\in \Uh $,
		\begin{equation}
			\label{eq:step1fully2}
			\left(\frac{\tu^{m+1}_h-u^m_h}{\Delta t},v\right)+ b(\tu^m_h,\tu^{m+1}_h,v)+\mu(\Grad \tu_h^{m+1},\Grad v)-(p_h^m,\Div v)= (\PUh f^{m}, v),
		\end{equation}
		\item[{\bf Step 2}] (Projection step): Next we seek $(u^{m+1}_h,p^{m+1}_h)\in \Yh\times\Ph$ which satisfy for all $(v,q)\in \Yh\times\Ph$,
		\begin{subequations}
			\label{eq:projectionfullydiscrete2}
			\begin{align}
				\label{eq:projection12}
				\left(\frac{u^{m+1}_h-\tu^{m+1}_h}{\Delta t }, v\right)&= -(\Grad( p^{m+1}_h-p_h^m), v),\\
				\label{eq:projection22}
				\left( u^{m+1}_h,\Grad q\right)&=0.
			\end{align}	
		\end{subequations}
	\end{enumerate}
\begin{remark}\label{rem:PVhonUh}
	We note that restricted to $\Yh$, $\PVh$ corresponds exactly to the projection onto $\mathbb{H}_h= \ker(C_h)$, where the operator $C_h:\Yh\to \Ph$ is defined as in~\cite{Guermond1998}, by $(C_h u,q)=(u,\Grad q)$ for all $q\in \Ph$, i.e., one writes~\eqref{eq:projection12}--\eqref{eq:projection22} as
	\begin{equation*}
		\begin{split}
			\frac{u^{m+1}_h-\tu^{m+1}_h}{\Delta t } & = -C_h^\top (p^{m+1}_h-p_h^m),\\
			C_h u^{m+1}_h & = 0.
		\end{split}
	\end{equation*}
		 Thus $\PVh$ is exactly the projection step~\eqref{eq:projectionfullydiscrete2} of the numerical scheme, when restricted to $\Yh$. In particular, we have $\PVh \tu_h^{m+1} = u^{m+1}_h$. 
\end{remark}
\begin{remark}[Formulation of projection step as a Poisson problem]
	The projection step can be rewritten as a Poisson problem as follows: Taking $\Grad q\in \Grad\Ph$ as a test function in~\eqref{eq:projection12} and using the weak divergence constraint~\eqref{eq:projection22}, we see that $p_h^{m+1}\in \Ph$ can be computed from
	\begin{equation*}
		(\Grad p^{m+1}_h,\Grad q) = (\Grad p^m_h,\Grad q)-\left(\frac{\Div \tu^{m+1}_h}{\Delta t},q\right),
	\end{equation*}
	and then $u_h^{m+1}$ can be computed via~\eqref{eq:projection12}.
	
\end{remark}
The solvability of the scheme~\eqref{eq:step1fully2}--\eqref{eq:projectionfullydiscrete2} is shown in the same way as for the nonincremental Darcy version of the scheme above in Lemma~\ref{lem:solvability}, alternatively see also~\cite[Lemma 2.1]{Eymard2024}. Furthermore, the scheme satisfies the discrete energy estimate
\begin{lemma}[Discrete energy estimate~\cite{Eymard2024}]
	\label{lem:discenergy2}
	The approximations computed by the scheme~\eqref{eq:step1fully2}--\eqref{eq:projectionfullydiscrete2} satisfy for any integer $M\in \left\{0,1,\dots, \frac{T}{\Delta t}\right\}$
	\begin{equation}
		\label{eq:energybalance2}
		E^M_h + \frac12 \sum_{m=0}^{M-1} \norm{\tu_h^{m+1} - u_h^m}_{L^2(\dom)}^2 +\mu\Delta t \sum_{m=0}^{M-1} \norm{\Grad \tu_h^{m+1}}_{L^2(\dom)}^2  = E_h^0 + \Delta t \sum_{m=0}^{M-1} (f^m,\tu_h^{m+1}),
	\end{equation}
	where 
	\begin{equation*}
		E_h^M := \frac12 \norm{u_h^M}_{L^2(\dom)}^2 + \frac{\Delta t^2}{2}\norm{\Grad p_h^{M}}_{L^2(\dom)}^2.
	\end{equation*}
	
\end{lemma}
\begin{proof}
	This was proved in~\cite[Lemma 3.1]{Eymard2024}.
\end{proof}
\begin{remark}
	Note that the interpolation operator $\Pi_N$ as defined in~\cite{Eymard2024} was not needed to prove this lemma in~\cite{Eymard2024}, it was only to prove the convergence of the scheme.
\end{remark}
As in Remark~\ref{rem:fneq0}, this results in a uniform in $h,\Delta t >0$ bound on the energy after applying the discrete Gr\"onwall inequality, Lemma~\ref{lem:discretegronwall}. 
Next, we define the piecewise constant interpolations in time of the approximations $u_h^m$, $\tu_h^m$ and $p_h^m$ by
\begin{alignat}{2}
	u_h(t)& = u_h^m,\qquad &t\in [t^m,t^{m+1}),\label{eq:defuh2}\\
	\widetilde{u}_h(t) &= \widetilde{u}_h^{m+1},\qquad &t\in (t^m,t^{m+1}],\\
	p_h(t)& = p_h^{m+1},\qquad & t\in (t^m,t^{m+1}],\\
	f_h(t)& = \frac{1}{\Delta t}\int_{t^m}^{t^{m+1}}\PUh f(s)ds,\qquad &t\in [t^m,t^{m+1}),\label{eq:defph2}
\end{alignat}
for $m=0,1,2,\dots$ with
\begin{equation*}
	\tu_h(r)=u^0_h,\quad \text{and}\quad p_h(r)=p_h^0,\quad \text{for }\, r\leq 0.
\end{equation*}

\subsection{Convergence of the numerical scheme}
For the proof of convergence of this scheme, we enounter similar issues as for the previous scheme that was based on a Darcy formulation of the projection step. On the one hand, we have uniform bounds on $\tu_h$ in $L^2(0,T;H^1_0(\dom))$ on the other hand, $u_h$ is approximately weakly divergence free, but we cannot say the same for $\tu_h$ directly. Furthermore, we have very weak bounds on the pressure approximation.

First, from the energy balance, Lemma~\ref{lem:discenergy2}, we obtain the following uniform a priori estimates for the sequences  $\{{u}_h\}_{h>0}$ and $\{\tu_h\}_{h>0}$:
\begin{align*}
	&\{{u}_h\}_{h>0}\subset L^\infty(0,T;L^2(\dom)),\\
	&\{\widetilde{u}_h\}_{h>0}\subset L^\infty(0,T;L^2(\dom))\cap L^2(0,T;H^1_0(\dom)).
\end{align*}
These uniform estimates imply, using the Banach-Alaoglu theorem that there exist weakly convergent subsequences, which, for the ease of notation, we still denote by $h\to 0$,
\begin{align}
{u}_h\weakstar {u},\quad	\widetilde{u}_h\weakstar \widetilde{u},&\quad \text{in }\, L^\infty(0,T;L^2(\dom)),	\label{eq:uhweakconv2}\\
	\widetilde{u}_h\weak \widetilde{u},&\quad \text{in }\, L^2(0,T;H^1_0(\dom)).\label{eq:Hhweakconv2}
\end{align}
Due to the discrete energy estimate~\eqref{eq:energybalance2}, Lemma~\ref{lem:samelimits} still applies to the sequences $\{{u}_h\}_{h>0}$ and $\{\tu_h\}_{h>0}$ and they have the same weak limit $u\in L^2(0,T;H^1_0(\dom))\cap L^\infty(0,T;L^2(\dom))$. It also follows just as in Lemma~\ref{lem:divfree2} that the limit is weakly divergence free. However, in this case, we cannot conclude that piecewise linear interpolations in time of the approximations $u^m_h$, $\tu_h^m$ converge to the same limit due to the lack of the corresponding bounds in the energy inequality. Therefore, we provide an estimate on the time differences of $\tu_h$ instead of $\partial_t u_h$, similar to~\cite{Eymard2024}. This will be one of the ingredients to apply the Simon compactness result~\cite{Simon1987}, restated in Theorem~\ref{thm:Simonlemma}, in order to derive compactness of the approximating sequences in $L^2([0,T]\times\dom)$.

\begin{lemma}
	\label{lem:timecontinuityutilde}
	Denote $Z = (H^1_{\Div}(\dom)\cap H^s(\dom))^*$. Assume that $h^s\leq C\Delta t$, where $2\leq s\leq k+1$ and $k$ is the polynomial degree of $\Uh$. Then the aproximations $\tu_h$ computed by~\eqref{eq:step1fully2}--\eqref{eq:projectionfullydiscrete2} satisfy for any $0\leq \tau <T$,  
	\begin{equation}
		\label{eq:timecontutildeweak}
		\int_0^{T-\tau}\norm{\tu_h(t+\tau)-\tu_h(t)}_Z^2 dt \leq C \tau (\tau+\Delta t).
	\end{equation}

\end{lemma}
\begin{proof}
	Given $\tau>0$ and $t\in (0,T-\tau]$, let $m_1\in \N$ such that $\tu_h(t) = u_h^{m_1}$ and $ m_2\in \N$ such that $\tu_h(t+\tau) = \tu_h^{m_2+1}$. If $\tau<\Delta t$, there is nothing to prove, so we can assume without loss of generality that $m_2\geq m_1$. Then
	\begin{equation*}
\tu_h(t+\tau)-\tu_h(t) = \sum_{m=m_1}^{m_2} (\tu_h^{m+1}-\tu_h^m).
	\end{equation*}
	We take a test function $v\in H^1_{\Div}(\dom)\cap H^s(\dom)$,  and plug in the scheme \eqref{eq:step1fully2}--\eqref{eq:projectionfullydiscrete2},
	\begin{align*}
		(\tu_h(t+\tau)-\tu_h(t),v) &= \sum_{m=m_1}^{m_2} (\tu_h^{m+1}-\tu_h^m,v)\\
		& = \sum_{m=m_1}^{m_2} (\tu_h^{m+1}-\tu_h^m,\PUh v)\\
		& = - \Delta t  \sum_{m=m_1}^{m_2} \underbrace{b(\tu_h^m,\tu_h^{m+1},\PUh v)}_{\text{I}} - \Delta t  \sum_{m=m_1}^{m_2} \underbrace{\mu(\Grad \tu_h^{m+1},\Grad \PUh v)}_{\text{II}} \\
		&\quad  - \Delta t  \sum_{m=m_1}^{m_2} \underbrace{(2\Grad p_h^{m}-\Grad p_h^{m-1},\PUh v)}_{\text{III}} + \Delta t  \sum_{m=m_1}^{m_2} \underbrace{(f^m,\PUh v)}_{\text{IV}} .
	\end{align*}
	We estimate each of the terms I -- IV, similar to what we did in Lemma~\ref{lem:timederivativebounds}:
	\begin{align*}
		|\text{I}|&=\left|b(\tu_h^m,\tu_h^{m+1},\PUh v)\right|\\
		& \leq \norm{\tu^m_h}_{L^2} \norm{\Grad\tu_h^{m+1}}_{L^2}\norm{\PUh v}_{L^\infty}+\frac12\norm{\tu^{m+1}_h}_{L^2} \norm{\Div\tu_h^{m}}_{L^2}\norm{\PUh v}_{L^\infty}\\
			& \leq C \left(\norm{\tu^m_h}_{L^2} \norm{\Grad\tu_h^{m+1}}_{L^2}+\norm{\tu^{m+1}_h}_{L^2} \norm{\Grad\tu_h^{m}}_{L^2}\right)\norm{ v}_{L^\infty}\\
				& \leq C \left(\norm{\tu^m_h}_{L^2} \norm{\Grad\tu_h^{m+1}}_{L^2}+\norm{\tu^{m+1}_h}_{L^2} \norm{\Grad\tu_h^{m}}_{L^2}\right)\norm{ v}_{H^2},
	\end{align*}
	using~\eqref{eq:Lr}.
	For the second term, we have
	\begin{align*}
		|\text{II}|
		& \leq\mu \norm{\Grad \tu^{m+1}_h}_{L^2(\dom)}\norm{\Grad\PUh v}_{L^2(\dom)}\\
		& \leq C \norm{\Grad \tu^{m+1}_h}_{L^2(\dom)}\norm{ v}_{H^1(\dom)},
	\end{align*}
	using~\eqref{eq:H1}.
	For the third term, III, we have, using that $v$ is divergence free,  the properties of the $L^2$-projection,~\eqref{eq:L2projproperties}, and that by the energy estimate $\norm{\Grad p}_{L^\infty(0,T;L^2(\dom))}\leq C \Delta t^{-1}$,
 \begin{align*}
		|\text{III}|&\leq C(\norm{\Grad p^m_h}_{L^2(\dom)}+\norm{\Grad p^{m-1}_h}_{L^2(\dom)})\norm{\PUh v-v}_{L^2( \dom)}\\
		& \leq C(\norm{\Grad p^m_h}_{L^2(\dom)}+\norm{\Grad p^{m-1}_h}_{L^2(\dom)})h^s\norm{ v}_{H^s( \dom)}\\
		& \leq C \frac{h^s}{\Delta t}\norm{ v}_{H^s( \dom)}\\
		& \leq  C \norm{ v}_{H^s( \dom)},
	\end{align*}
	under the condition that $h^s\leq C {\Delta t}$.
	Finally, for the last term,  IV, we have
	\begin{equation*}
		|\text{IV}|  \leq \norm{f^m}_{ L^2(\dom)}\norm{v}_{ L^2(\dom)} .
	\end{equation*}
	Combining the estimates I -- IV, we have, after taking the supremum over $v\in H^1_{\Div}(\dom)\cap H^s(\dom)$,
	\begin{align*}
		\norm{\tu_h(t+\tau)-\tu_h(t)}_Z& \leq C \Delta t  \sum_{m=m_1}^{m_2} \Big( \norm{\tu^m_h}_{L^2} \norm{\Grad\tu_h^{m+1}}_{L^2}+\norm{\tu^{m+1}_h}_{L^2} \norm{\Grad\tu_h^{m}}_{L^2}\\
		&\quad\hphantom{\leq C \Delta t  \sum_{m=m_1}^{m_2} \Big(} + \norm{\Grad \tu^{m+1}_h}_{L^2(\dom)}+ 1+ \norm{f^m}_{ L^2(\dom)} \Big).
	\end{align*}
	We square this identity and use that $\Delta t (m_2-m_1) \leq \tau +\Delta t$,
	\begin{align*}
			\norm{\tu_h(t+\tau)-\tu_h(t)}_Z^2& \leq C (\tau+\Delta t)\Delta t  \sum_{m=m_1}^{m_2} \Big( \norm{\tu^m_h}_{L^2} \norm{\Grad\tu_h^{m+1}}_{L^2}+\norm{\tu^{m+1}_h}_{L^2} \norm{\Grad\tu_h^{m}}_{L^2}\\
			&\quad\hphantom{\leq C (\tau+\Delta t)\Delta t  \sum_{m=m_1}^{m_2} \Big( } + \norm{\Grad \tu^{m+1}_h}_{L^2(\dom)}+ 1+ \norm{f^m}_{ L^2(\dom)} \Big).
	\end{align*}
	and then integrate this identity over time to obtain, using the energy estimate, Lemma~\ref{lem:discenergy2},  
	\begin{equation*}
	\begin{split}
			&\int_0^{T-\tau}\norm{\tu_h(t+\tau)-\tu_h(t)}^2_Z  dt  \\
			& \leq C (\tau+\Delta t)\int_0^{T-\tau} \int_t^{t+\tau} \Big( \norm{\tu _h(s-\Delta t)}_{L^2}^2 \norm{\Grad\tu_h(s)}^2_{L^2}+\norm{\tu _h(s)}_{L^2}^2 \norm{\Grad\tu_h(s-\Delta t)}^2_{L^2}\\
			&\quad \hphantom{\leq C (\tau+\Delta t)\int_0^{T-\tau} \int_t^{t+\tau} \Big(}+ \norm{\Grad \tu_h}^2_{L^2(\dom)}+ 1+ \norm{f_h}^2_{ L^2(\dom)} \Big) ds dt\\
			& \leq C (\tau+\Delta t)\tau \int_0^{T-\tau} \Big( \norm{\tu _h(s-\Delta t)}_{L^2}^2 \norm{\Grad\tu_h(s)}^2_{L^2}+\norm{\tu _h(s)}_{L^2}^2 \norm{\Grad\tu_h(s-\Delta t)}^2_{L^2}\\
			&\quad \hphantom{\leq C (\tau+\Delta t)\int_0^{T-\tau}  \Big(}+ \norm{\Grad \tu_h}^2_{L^2(\dom)}+ 1+ \norm{f_h}^2_{ L^2(\dom)} \Big) ds \\
			& \leq C (\tau+\Delta t)\tau. 
	\end{split}
	\end{equation*}
\end{proof}

Next, we use Lemma~\ref{lem:lionslike} to show uniform time continuity of $\tu_h$ in $L^2([0,T]\times\dom)$ which we can then use together with Simon's version of the Aubin-Lions-Simon lemma~\cite[Theorem 1]{Simon1987}, stated in Theorem~\ref{thm:Simonlemma}, to conclude precompactness of the sequence $\{\tu_h\}_{h>0}$ in $L^2([0,T]\times\dom)$. The following lemma and its proof closely follow~\cite[Lemma 3.6]{Eymard2024}. We include it here both for completeness and to illustrate that it can be established using Lemma~\ref{lem:lionslike}, instead of the version employed in~\cite{Eymard2024}.
\begin{lemma}\label{lem:timecont}
	The approximations $\tu_h$ computed by the scheme~\eqref{eq:step1fully2}--\eqref{eq:projectionfullydiscrete2} satisfy
	\begin{equation}
		\label{eq:timecont2}
		\int_0^{T-\tau}\norm{\tu_h(t+\tau)-\tu_h(t)}_{L^2(\dom)}^2 dt \longrightarrow 0,\quad \text{as }\, \tau\to 0,
	\end{equation}
	uniformly with respect to $h,\Delta t>0$.
\end{lemma}
\begin{proof}
	We recall from Remark~\ref{rem:PVhonUh} that ${u}_h$ is given by the projection of $\tu_h$ onto $\Vh$, specifically ${u}_h^m = \PVh \tu_h^m$ for all $m\geq 0$ and therefore also that $\PVh u_h = u_h$. Then, we can write $\tu_h = \tu_h-u_h-\PVh \tu_h +\PVh u_h + \PVh \tu_h$ and estimate the quantity in~\eqref{eq:timecont2} as follows, using triangle inequality and that $\PVh$ is an orthogonal projection:
	\begin{align*}
	&\int_0^{T-\tau}\norm{\tu_h(t+\tau)-\tu_h(t)}_{L^2(\dom)}^2 dt\\
	& \leq 3\int_0^{T-\tau}\norm{\tu_h(t+\tau)-\tu_h(t)-(u_h(t+\tau)-u_h(t))}_{L^2(\dom)}^2\\
	&\quad  + 3\int_0^{T-\tau}\norm{\PVh( \tu_h(t+\tau)-\tu_h(t))-\PVh( u_h(t+\tau)-  u_h(t))}_{L^2(\dom)}^2dt \\
	&\quad +3\int_0^{T-\tau}\norm{\PVh( \tu_h(t+\tau)- \tu_h(t))}_{L^2(\dom)}^2 dt\\
	& \leq 6	\underbrace{\int_0^{T-\tau}\norm{\tu_h(t+\tau)-\tu_h(t) - ({u}_h(t+\tau)-{u}_h(t))}_{L^2(\dom)}^2 dt}_{(i)} \\
	&\quad +  3\underbrace{\int_0^{T-\tau}\norm{ \PVh({\tu}_h(t+\tau)- {\tu}_h(t))}_{L^2(\dom)}^2 dt}_{(ii)}.
	\end{align*}
	We let $\epsilon>0$ arbitrary and given. For the first term, we have, using that $\norm{\tu_h-{u}_h}_{L^2}^2\leq C \Delta t$  by (an adaption of) Lemma~\ref{lem:samelimits}, that there is $h_1>0$ (and corresponding $\Delta t $) small enough such that for any $0<h<h_1$, we have
	\begin{equation*}
		(i) \leq 2\int_0^{T-\tau}\norm{\tu_h(t+\tau)  - {u}_h(t+\tau) }_{L^2(\dom)}^2 dt + 2\int_0^{T}\norm{\tu_h(t)-{u}_h(t)}_{L^2(\dom)}^2 dt < \frac{\epsilon}{8}.
	\end{equation*}
	For the second term (ii), we use Lemma~\ref{lem:lionslike}   combined with the weak time continuity of $\tu_h$, Lemma~\ref{lem:timecontinuityutilde}, the discrete energy inequality, Lemma~\ref{lem:discenergy2} to estimate for any $\eta>0$ and any $0<h<\min\{h_\eta,h_1\}$,
	\begin{equation*}
	\begin{split}
			(ii)& \leq \eta \int_0^{T-\tau} \norm{\tu_h(t+\tau)-\tu_h(t)}_{H^1_0}^2 dt + C_\eta \int_0^{T-\tau}\norm{\tu_h(t+\tau)-\tu_h(t)}_Z^2 dt\\
			& \leq \eta C_{E_0,f} + C_\eta \int_0^{T-\tau}\norm{\tu_h(t+\tau)-\tu_h(t)}_Z^2 dt
	\end{split}
	\end{equation*}
	where $Z:= (H^1_{\Div}(\dom)\cap H^s(\dom))^*$. Now we first choose $\eta = \epsilon / (24C_{E_0,f})$ and then since we know from Lemma~\ref{lem:timecontinuityutilde} that the second term goes to zero as $\tau\to 0$, we have that for any $0<\tau<\bar{\tau}$ small enough that
	\begin{equation*}
		C_\eta\int_0^{T-\tau}\norm{\tu_h(t+\tau)-\tu_h(t )}_Z^2 dt \leq \frac{\epsilon}{24}.
	\end{equation*}
	Thus, combining these estimates, we obtain that for any
	 $0<h<\min\{h_\eta,h_1\}$ and $0<\tau< \bar{\tau}$, 
	\begin{equation*}
			\int_0^{T-\tau}\norm{\tu_h(t+\tau)-\tu_h(t)}_{L^2(\dom)}^2 dt<\epsilon,
	\end{equation*}
	 which proves the result since $\epsilon$ was arbitrary.
\end{proof}
The result of this lemma, combined with the uniform $L^2([0,T];H^1_0(\dom))$-bound on the $\tu_h$ coming from the energy estimate, Lemma~\ref{lem:discenergy2}, allow us to apply Simon's version of the Aubin-Lions-Simon lemma, Theorem~\ref{thm:Simonlemma}, to conclude that a subsequence of $\{\tu_h\}_{h>0}$ converges strongly in $L^2([0,T]\times\dom)$ to the limit $u$. By Lemma~\ref{lem:samelimits}, this implies that also $\{{u}_h\}_{h>0}$ has a strongly convergent subsequence to the same limit $u$. Then the fact that $u$ is a Leray-Hopf solution of the incompressible Navier-Stokes equations follows along the lines of Theorem~\ref{thm:convergence}, or one can also check that Step 2 of Theorem 3.8 in~\cite{Eymard2024} applies.

\section*{Acknowledgments}
I would like to thank Alexandre Chorin and Saleh Elmohamed for helpful discussions and providing useful references on this topic.

\appendix
\section{Discrete Gr\"onwall inequality}
The following standard lemma is useful for proving an energy bound in the case that the external source in the fluid equation is nonzero.
\begin{lemma}
	\label{lem:discretegronwall}
	Let $\{a_n\}_{n\in \N}\geq 0$, and $\{b_n\}_{n\in \N}\geq 0$ be two nonnegative sequences satisfying
	\begin{equation}\label{eq:assgronwall}
		a_n \leq (1+\mu \Delta t)(a_{n-1} +\beta b_{n-1}),
	\end{equation}
	for parameters $\mu,\Delta t, \beta\geq 0$.
	Then $a_n$ satisfies
	\begin{equation}\label{eq:discgronwall}
		a_n \leq \exp(\mu \Delta t n)a_0 +\beta \sum_{m=0}^{n-1}\exp(\mu\Delta t (n-m))b_m,
	\end{equation}
	and in particular,
	\begin{equation*}
		a_n \leq \exp(\mu \Delta t n)\left(a_0+\beta \sum_{m=0}^{n-1}\ b_m \right).
	\end{equation*}
\end{lemma}
\begin{proof}
	We can prove this by induction on $n$. Clearly, it is true for $n=0$. For the induction step, we assume the statement holds for $a_{n-1}$ and deduce that it also holds for $a_n$. Thus, we take~\eqref{eq:assgronwall} for $a_n$ and plug in~\eqref{eq:discgronwall} for $n-1$:
	\begin{align*}
		a_n & \leq (1+\mu \Delta t)(a_{n-1} +\beta b_{n-1})\\
		& \leq (1+\mu \Delta t)\left(\exp(\mu \Delta t (n-1))a_0 +\beta \sum_{m=0}^{n-2}\exp(\mu\Delta t (n-1-m))b_m+\beta b_{n-1}\right)\\
		& \leq \exp(\mu\Delta t )\left(\exp(\mu \Delta t (n-1))a_0 +\beta \sum_{m=0}^{n-1}\exp(\mu\Delta t (n-1-m))b_m\right)\\
		& \leq\exp(\mu \Delta t n)a_0 +\beta \sum_{m=0}^{n-1}\exp(\mu\Delta t (n-m))b_m,
	\end{align*} 
	where we used that for positive $x$, we have $1+x\leq \exp(x)$. The second statement follows by upper bounding $\exp(\mu\Delta t(n-m))$ by $\exp(\mu\Delta t n)$ in the sum above.
\end{proof}

\section{Aubin-Lions-Simon lemma}
\label{app:ALSlemma}

We recall Simon's version of the Aubin-Lions-Simon lemma~\cite[Theorem 1]{Simon1987}:
\begin{theorem}[Simon's compactness criterion~\cite{Simon1987}]
	\label{thm:Simonlemma}
	Let $F\subset L^p(0,T;B)$ where $1\leq p<\infty$ and $B$ is a Banach space. Then $F$ is relatively compact in $L^p(0,T;B)$ if and only if the following two conditions are satisfied:
	\begin{enumerate}
		\item $\left\{\int_{t_1}^{t_2} f(t) dt \, : \, f\in F\right\}$ is relatively compact in $B$ for all $0<t_1<t_2<T$,  \label{cond1} 
		
		\item $\norm{f(\cdot + \tau)-f}_{L^p(0,T-\tau;B)}\longrightarrow 0$ as $\tau\to 0$, uniformly for $f\in F$.   \label{cond2}
	\end{enumerate}
	
\end{theorem}

\bibliographystyle{abbrv}
\bibliography{projection}

\end{document}